\documentclass[11pt,reqno]{amsart}
\usepackage{amsfonts}
\usepackage{amssymb,color}
\usepackage{amssymb}
\usepackage{setspace}
\usepackage[pagewise]{lineno}

\usepackage{hyperref}
\hypersetup{colorlinks=true,linkcolor=black, urlcolor=blue, citecolor=blue}

\usepackage[margin=2.8cm, a4paper]{geometry}
\usepackage{xcolor}

\def\normo#1{\left\|#1\right\|}

\def\norm#1{\|#1\|}

\newcommand{\R}{{\mathbb R}}

\newcommand{\Z}{{\mathbb Z}}
\newcommand{\X}{{\mathbb X}}

\newcommand{\les}{{\lesssim}}

\def\norm#1{\|#1\|}
\def\normo#1{\left\|#1\right\|}

\newcommand{\p}{\partial}

\newcommand{\EQ}[1]{\begin{equation}\begin{split} #1 \end{split}\end{equation}}
\newcommand{\EQN}[1]{\begin{equation*}\begin{split} #1 \end{split}\end{equation*}}

\newcommand{\Del}[1]{}

\numberwithin{equation}{section}

\newtheorem{theorem}{Theorem}[section]

\newtheorem{lemma}[theorem]{Lemma}
\newtheorem{corollary}[theorem]{Corollary}

\theoremstyle{remark}

\theoremstyle{definition}
\newtheorem{definition}[theorem]{Definition}

\newtheorem{thm}{Theorem}[section]

\newtheorem{prop}[thm]{Proposition}

\newtheorem{rem}[thm]{Remark}

\numberwithin{equation}{section}\allowdisplaybreaks

\def\leq{\leqslant}

\def\leq{\leqslant}
\def\geq{\geqslant}

\def\no{\nonumber}

\def\Z{{\mathbb{Z}}}

\begin{document}

\title{
On the well-posedness of the compressible Navier-Stokes equations
}
\author[Z. Guo]{Zihua Guo}
\address{School of Mathematics, Monash University, Clayton VIC 3800, Australia}
\email{zihua.guo@monash.edu}
\author[M. Yang]{Minghua Yang}
\address{School of Information Management and Mathematics, Jiangxi University of Finance and
Economics, Nanchang, 330032, China}
\email{minghuayang@jxufe.edu.cn}
\author[Z. Zhang]{Zeng Zhang}
\address{School of Mathematics and statistics, Wuhan University of Technology, Wuhan, 430070, China}
\email{zhangzeng534534@whut.edu.cn}
\thispagestyle{empty}

\begin{abstract}
We consider the Cauchy problem to the barotropic compressible Navier-Stokes equations. We obtain optimal local well-posedness in the sense of Hadamard in the critical Besov space $\X_p=\dot{B}_{p,1}^{d/p}\times \dot{B}_{p,1}^{-1+d/p}$ for $1\leq p<2d$ with $d\geq2$. The main new result is the continuity of the solution map from $\X_p$ to $C([0,T]: \X_p)$. In the previous works \cite{D2001, D2005, D2014}, the only known continuity of the solution map was from $\X^{-1}_p$ to $C([0,T]: \X^{-1}_p)$ for $1\leq p<d$ as a direct consequence of the uniqueness argument. 

To prove our results, we first extend the method of frequency envelope (see \cite{Tao04}) to the transport-parabolic setting. By this we obtained enhanced uniform estimates, in particular the uniform high-frequency smallness, for the sequence of smooth approximating solutions. Then we use the Lagrangian approach for the compressible Navier-Stokes equations (see \cite{D2014}) to derive a new difference estimate for the velocity in $L_t^1L_x^\infty$, which allows us to control the difference for low-frequency.  As a by-product of our results, the Lagrangian transform $(a,u)\to (\bar a, \bar u)=(a\circ X, u\circ X)$ used in \cite{D2014} is a continuous bijection and hence bridges the Eulerian and Lagrangian methods.
\end{abstract}


\maketitle

\tableofcontents

\section{Introduction} \label{intro}

In this paper, we consider the Cauchy problem to the compressible Navier-Stokes equations, which describe the motion of a barotropic fluid in the whole space $\mathbb{R}^d$ for $d \geq 2$:
\begin{equation}\label{E1.1}
\left\{
\begin{aligned}
& \partial_{t}\rho+{\rm div}\, (\rho u)=0,\quad  t>0,\,  x\in \mathbb{R}^{d},\\
&  \partial_{t}(\rho u)+{\rm div}\, (\rho u\otimes u)+\nabla P(\rho)=\mathcal{A}u,\quad  t>0,\,  x\in \mathbb{R}^{d},\\
& \rho(0,x)=\rho_{0}(x),\,   u(0,x)=u_{0}(x),\quad x\in \mathbb{R}^{d}.  \\
\end{aligned}
\right.
\end{equation}
Here, the two unknown functions $\rho(t,x)$ and $u(t,x)$ represent the fluid's density and velocity field, respectively. $P(\rho)$ represents the pressure of the fluid that depends on the density and $\mathcal{A} = \mu\Delta + (\mu + \lambda)\nabla \mathrm{div}$ is the Lamé operator, representing the viscosity, where the constants $\mu$ and $\lambda$ satisfy the elliptic conditions $\mu > 0$ and $2\mu +\lambda > 0$.

We assume that the fluid density $\rho$ is a small perturbation near $1$. Let $a=\rho-1$. To simplify our presentation, we assume that $$\mu=P^{\prime}(1)=1,  \,I(a) = \frac{a}{1 + a}, \,G'(a) = \frac{P'(1+ a)}{1+ a},$$
which allows us to rewrite \eqref{E1.1} as:
\begin{equation}\label{eulercauchy}
\left\{
\begin{aligned}
& a_t+u\cdot \nabla a=-(1+a){\rm div}\, u,\,  t>0,\,  x\in \mathbb{R}^{d},\\
&  \partial_{t}u-\mathcal{A}u=-u\cdot \nabla u-I(a)\mathcal{A}u-\nabla G(a),\,  t>0,\,  x\in \mathbb{R}^{d},
\\
& a(0,x)=a_{0}(x),\,   u(0,x)=u_{0}(x),\, x\in \mathbb{R}^{d}.  \\
\end{aligned}
\right.
\end{equation}
The Cauchy problem \eqref{eulercauchy} is of fundamental importance and has been extensively studied. We refer to \cite{PL1998, RH2015} and the references therein for the detailed introductions.  

We are interested in the well-posedness of \eqref{eulercauchy} in the sense of Hadamard.  First we recall the standard definition. 

\begin{definition}[Well-posedness]  
The Cauchy problem \eqref{eulercauchy} is said to be locally well-posed in a metric space $\X$ if for any $(a_{01},u_{01})\in \X$, there exists a neighbourhood $U$ of $(a_{01},u_{01})$ and $T=T(U)>0$ such that for any $(a_0,u_0)\in U$ we have
\begin{itemize}
\item[(i)] ({\bf Existence}): There exists a distributional solution $(a,u)\in E\subset C([0,T):\X)$ to \eqref{eulercauchy} for some set $E$.
\item[(ii)] ({\bf Uniqueness}): The solution $(a,u)$ is unique in $E$.
\item[(iii)] ({\bf Continuity}): The solution map $(a_0,u_0)\to (a,u):=S_T(a_0,u_0)$ is continuous from $U$ to $C([0,T):\X)$. 
\end{itemize} 
If the above holds for any $T>0$, then we say \eqref{eulercauchy} is globally well-posed in $\X$.
\end{definition}

\begin{rem}
The three properties are important ingredients of well-posedness in the sense of Hadamard.
For the continuity of the solution map, one may impose some weak version, e.g. estimate the difference in weaker topology. Here we employ the strong version, namely the space for the difference is the same as the initial data, so that the equation defines a continuous flow. In \cite{Tao2}, in order to extend the solution beyond blowup, Tao introduced a notion of semi-strong class solutions so that for the mass-critical Schr\"odinger equation one has existence, uniqueness and some weak continuous dependence (in this notion continuous dependence fails).
Here we give another example which shows that one has global existence, uniqueness, and even continuity in time, but the solution map fails to be continuous. Consider the equation
\EQ{\label{eq:ODE}
\partial_t \hat u(\xi,t)=iV(|\hat u(\xi,t)|^2)\hat u(\xi,t)
}
where $V:\R\to \R$, in $L^\infty$ but not continuous. $\hat u$ is Fourier transform of $u$. 
Then we see $|\hat u(\xi,t)|^2$ is preserved under the flow. Thus we can solve \eqref{eq:ODE} explicitly and get
\[\hat u=e^{it V(|\hat u(\xi,0)|^2)}\hat {u}(\xi,0).\]
We have global existence and uniqueness in $H^s$, even $u\in C_t H^s$. However, the solution map $u(x,0)\to u$ is not continuous from $H^s$ to $C_tL^2$.
\end{rem}

In seeking local or global-in-time results, scaling invariance has emerged as a crucial consideration. This approach can be traced back to the pioneering work of Fujita and Kato \cite{FK1964} on the classical incompressible Navier-Stokes equations. For barotropic fluids, it can be verified that under the scaling transformation
\begin{equation}\label{critical}
\left\{
\begin{aligned}
&(a_{0}, u_{0})\rightsquigarrow(a_{0}(\iota x), \iota u_{0}(\iota x)),\\
&  (a(t,x), u(t,x))\rightsquigarrow(a(\iota^{2}t, \iota x), \iota u(\iota^{2}t, \iota x)), \iota>0,\\
\end{aligned}
\right.
\end{equation}
\eqref{eulercauchy} is invariant, provided that the pressure term has been changed accordingly. A natural choice of $\X$ which is scaling invariant for \eqref{eulercauchy} is the critical Besov space
\begin{equation}\label{criticalspace}
(a_{0},\, u_{0})\in \X_p=\dot{B}_{p,1}^{\frac{d}{p}}\times\dot{B}_{p,1}^{-1+\frac{d}{p}}.
\end{equation}
We also denote $\X_p^s=\dot{B}_{p,1}^{s+\frac{d}{p}}\times\dot{B}_{p,1}^{-1+s+\frac{d}{p}}.$
There are much literature studying \eqref{eulercauchy} with data in $\X_p$. 
In \cite{D2001,D2005}, Danchin established local existence for $1\leq p<2d$ and uniqueness for $1\leq p\leq d$. The uniqueness for $d<p<2d$ was not clear until \cite{D2014} where Danchin developed a Lagrangian approach for \eqref{eulercauchy}. After performing a Lagrangian transform $(\bar a,\bar u)=(a\circ X, u\circ X)$, where
\EQ{
X(t,x)=x+\int_0^t u(\tau,X(\tau,x))\mathrm{d}\tau.
}
Danchin \cite{D2014} showed that the new equation for $(\bar a, \bar u)$ can be handled using the contraction mapping principle. 

On the other hand, the range $p<2d$ is optimal for well-posedness in $\X_p$. Chen-Miao-Zhang \cite{CMZ2015} proved the ill-posedness in the sense that the continuity of the solution map $S_T$ fails at the origin in $\X_p$ when $p>2d$. Iwabuchi-Ogawa \cite{IO2022} proved the same result for $p=2d$.  In view of these results, it is natural to ask whether the continuity of the solution map $S_T$ holds in $\X_p$ for $1\leq p<2d$ and thus completes the theory of well-posedness in $\X_p$ in the sense of Hadamard. To the best of our knowledge, this was not addressed in the previous works.  Only some weak continuous dependence results were known as a by-product of the uniqueness argument. For example, in \cite{D2001, D2005} and  Proposition~\ref{diff-1-local} in appendix ~\ref{proofapelwp}, it was proved that for $1\leq p<d$
\begin{equation}\label{eq:diffest}
\begin{aligned}
\|S_T(a_{01},u_{01})-S_T(a_{02},u_{02})\|_{C([0,T]: \mathbb{X}_p^{-1})}
\lesssim\|(a_{01}-a_{02},u_{01}-u_{02})\|_{\mathbb{X}_p^{-1}}.
\end{aligned}
\end{equation}
For the case $p \geq d$, it is not known whether the above difference estimate holds true since the regularity of $u$ is too negative (see Remark \ref{remark1}).  
The first equation of \eqref{eulercauchy} is a transport equation.  It is thus natural to estimate the difference in $\X_p^{-1}$, namely lower-regularity space as one does not expect Lipschitz continuity of $S_T$ to hold in $\X_p$. 

It seems to us that the Lagrangian approach of \cite{D2014} could not address the continuity in the Eulerian coordinate either, although it can prove its uniqueness. The main reason is that although we can have mutual control between $(a,u)$ and $(\bar a, \bar u)$, we could not get the difference estimate in $\X_p$ from the new variables $(\bar a, \bar u)$. Therefore, in the Eulerian coordinate, the only known results concerning the continuous dependence is the weak continuity \eqref{eq:diffest} that is restricted to $1\leq p<d$.

The purpose of this paper is to address the continuity of the solution map $S_T$ in $\X_p$ for $1\leq p<2d$ by introducing a new approach. We believe our approach may be applicable to many other problems in the field. Our main results are

\begin{theorem} [Local well-posedness]\label{lwp}
Assume $d\geq 2$ and $1 \leq p <2d$. Let $c>0$ be small enough (depending on $p, d$). Then for any $(a_{0}, u_{0})\in \X_p$ with $\|a_{0}\|_{\dot{B}_{p,1}^{{d/p}}}\leq c$, there exists a neighbourhood $U$ of $(a_{0}, u_{0})$ in $\X_p$ and $T=T(U)>0$, such that for any data $(\tilde{a}_{0}, \tilde{u}_{0}) \in U$, the Cauchy problem \eqref{eulercauchy} has a unique solution
\EQN{
(\tilde{a}, \tilde{u}):=S_{T}(\tilde{a}_0,\tilde{u}_0)\in Z_p(T):=C([0,T];\dot{B}_{p,1}^{\frac{d}{p}})\times \big(C([0,T];\dot{B}_{p,1}^{-1+\frac{d}{p}})\cap L^{1}([0,T]; \dot{B}_{p,1}^{1+\frac{d}{p}})\big).
}
Moreover, the solution map $S_T$ is continuous from $U$ to $Z_p(T)$.
\end{theorem}

\begin{rem}
The novelty of the above theorem is the continuity of $S_T$. 
A direct consequence is that the Lagrangian transform $(a,u)\to (\bar a,\bar u)=(a\circ X, u\circ X)$ is a continuous bijection from $C([0,T]:\X_p)$ to $C([0,T]:\X_p)$.
For simplicity, we only prove the continuity assuming that $\|a_{0}\|_{\dot{B}_{p,1}^{{d/p}}}$ sufficiently small. It seems to us that the smallness condition may be replaced by the positivity condition $\inf_{x\in \mathbb{R}^{d}} (1+a_{0})>0$ if combining our approach and the methods from \cite{CMZ2010R, D2007, D2014}. We do not pursue it in this paper. 
\end{rem}

\begin{rem}
In \cite{D2000}, Danchin established the global existence and uniqueness of \eqref{eulercauchy} with small data in $\X_2$ under some additional assumptions on the low frequency of $a_0$. Later, Charve-Danchin \cite{FD2010}, Chen-Miao-Zhang
\cite{CMZ2010} and Haspot \cite{H2011} obtained the global existence and uniqueness in some hybrid Besov spaces related to $\X_p$ for certain range of $p$. In a recent work \cite{GSY},  the first and second named authors together with Song proved global well-posedness in the optimal Besov space under some additional assumption on the low frequency of the density and momentum.  
\end{rem}

In the rest of the introduction, we would like to describe our approach to prove the continuity. Since we do not expect Lipschitz continuity, we have to use some non-perturbative methods. We are aware of two methods:

{\bf Bona-Smith method.} This method was introduced by Bona-Smith in \cite{Bona-Smith} for the KdV equation and was widely used in dispersive PDE.  It was used in \cite{GuLi21, GuLiYi18} for the incompressible Euler equations, and some continuity was proved in some critical Besov and Triebel-Lizorkin spaces. The main ingredients are 1) uniform boundedness; 2) the difference estimate \eqref{eq:diffest}; 3) an extra difference estimate: for $(a_0,u_0)\in U$
\EQ{\label{eq:diffest2}
\norm{S_T(a_0,u_0)-S_T(P_{\leq N} a_0, P_{\leq N}u_0)}_{C([0,T]:\X_p)}\les \norm{(a_0-P_{\leq N}a_0,u_0-P_{\leq N}u_0)}_{\X_p},
}
where $P_{\leq N}$ is some frequency projection operator defined in \eqref{PS}.  We could apply this method to \eqref{eulercauchy}, but we need to restrict the range of  $p$ to $1\leq p<d$ due to the restriction of \eqref{eq:diffest}.  See the appendix for the details.

{\bf Frequency envelope method.} Tao introduced this method in \cite{Tao04} and it was also widely used in dispersive PDE (see for instance \cite{KoTz}).  The main ingredients are 1) uniform boundedness; 2) the difference estimate \eqref{eq:diffest}; 3) uniform tail estimate on frequency: for $(a^n_0,u^n_0)\to (a_0,u_0)\in U$, then $\forall \varepsilon>0$, there exists $N>0$ such that
\EQ{
\norm{P_{>N}[S_T(a_0^n,u_0^n)]}_{C([0,T]: \X_p)}\leq \varepsilon, \quad \forall n.
}

In this paper, we will combine the frequency envelope method, and the Lagrangian approach in \cite{D2014} to replace \eqref{eq:diffest}, to obtain continuity for the full range of $p$: $1\leq p<2d$. More precisely, we prove Theorem~\ref{lwp}
in three steps: assume $(a_0,u_0)\in U$ and a sequence $(a_{0}^{n}, u_0^n)$
converges to $(a_{0},u_{0})$ in $\dot{B}_{p,1}^{\frac{d}{p}}\times \dot{B}_{p,1}^{-1+\frac{d}{p}}$ as well as $(a^n,u^n): =S_T(a_0^n,u_0^n)$.

{\bf Step 1.} Uniform boundedness and small tails estimate on frequency: for $1\leq p<2d$, there exists a $T>0$ such that
\EQ{\norm{(a^n,u^n)}_{Z_p(T)}\leq C, \quad \forall\, n}
and
$\forall \varepsilon>0$, there exists $N>0$ such that
\EQ{
\norm{(P_{>N}a^n,P_{>N}u^n)]}_{Z_p(T)}\leq \varepsilon, \quad \forall n.
} This will be proved using the frequency envelope method.

{\bf Step 2.} A new difference estimate of the velocity field. We show for $1\leq p<2d$, there holds
\EQ{
\norm{u-u^{n}}_{L^1([0,T];L^\infty)}\les \norm{(a_{0}^{n}-a_{0},u_{0}^{n}-u_{0})}_{\X_p}.
}
This will be proved using the Lagrangian approach in \cite{D2014}.

{\bf Step 3.} Difference estimate for the low-frequency component. For fixed large $N>0$, we need to derive the difference estimate for the sequence $(P_{\leq N} a^n, P_{\leq N} u^n)$ using the results in Steps 1 and 2.

The remainder of this paper is structured as follows. In Section~\ref{listestimate}, we list some well-known results and establish the key estimates needed for the proof of main results. Section~\ref{prooflwp} is dedicated to proving Theorem~\ref{lwp}. In the appendix, we prove the continuity of the solution maps using the Bona-Smith method.

\section{Auxiliary results }\label{listestimate}

Throughout this paper, the symbols $C, c, c_{i}, C_{i} (i\in \mathbb{Z})$  represent generic constants, which may vary from line to line, and the meaning of which will be clear from the context. In some places, we will use the notation $A \lesssim B$ to indicate that $A \leq CB$.

We use $\mathcal{S}$ and $\mathcal{S}'$ to denote the sets of Schwartz functions and tempered distributions over $\mathbb{R}^d$, respectively. The Fourier transform of a function $f$ is denoted by $\mathcal{F} f = \hat{f}=(2\pi)^{-\frac{d}{2}}\int_{\R^d}e^{-ix\xi}f(x)dx$, and its inverse Fourier transform is denoted by $\mathcal{F}^{-1} f = \check{f}$.  All function spaces are defined over $\mathbb{R}^d$. For simplicity, the domain will often be omitted; for example, we use $L^p$ instead of $L^p(\mathbb{R}^d)$ unless otherwise indicated.

Suppose $\varphi \in C^{\infty}(\R^d)$ satisfies $0\leq \varphi \leq 1$, $\varphi = 1$ on $B(1/2)$ and $\varphi = 0$ outside $B(1)$. Set $\psi(\xi) = \varphi(\xi/2) - \varphi(\xi)$, we denote  $\psi_j(\xi) = \psi(\xi/2^j)$ and $\varphi_j(\xi) = \varphi(\xi/2^j) $. The Littlewood-Paley frequency projection operator is defined by
\EQ{ \label{PS}
\dot{\Delta}_j:=(\mathcal{F}^{-1} \psi_j)*, \ \ \ \dot{S}_{j}=P_{\leq j}: = (\mathcal{F}^{-1}\varphi_j)*, \ \ \ P_{> j}: =I_{d} - \dot{S}_{j},
}
where $*$ is the convolution operator in $\R^d$. It is easy to see $\dot{S}_{j}f=\sum_{k\leq j-1}\dot{\Delta}_{k}f$ for $j\in \mathbb{Z}$.
Formally, any product
of two distributions $f$ and $g$ may be decomposed into (\cite{Bo81})
\begin{equation}\label{fg}
\begin{aligned}
fg=\dot{T}_fg +\dot{ T}_gf+ \dot{R}(f,g),
\end{aligned}
\end{equation}
where
\begin{equation}\label{TRfg}
\begin{aligned}
\dot{T}_fg:= \sum_{j}\dot{S}_{j-1}f \dot{\Delta}_{j}g , \quad \dot{R}(f,g):=\sum_{|j-k|\leq1}\dot{\Delta}_{j} f \dot{\Delta}_k g
=\sum_{j}\dot{\Delta}_j f \tilde{\dot{\Delta}}_{j}g.
\end{aligned}
\end{equation}
With our choice of $\varphi$, one can verify that
\begin{equation}\label{S1fg}
\begin{aligned}
\dot{\Delta}_{j}\dot{\Delta}_{k}f=0,\, \,|j-k|\geq2
\end{aligned}
\end{equation}
and
\begin{equation}\label{Sfg}
\begin{aligned}\dot{\Delta}_{j}(\dot{S}_{k-1}f\dot{\Delta}_{k}f)=0,\, \, |j-k|\geq5.
\end{aligned}
\end{equation}
Let $\mathcal{S}'\backslash \mathcal{P}$ denote the tempered distribution modulo the polynomials. For $1 \leq p, q \leq \infty, s\in\mathbb{R}$, denote
\begin{align}
\dot{B}^s_{p,q} := \big \{f\in \mathcal{S}'\backslash \mathcal{P}:\ \  \|f\|_{\dot{B}^s_{p,q}} < \infty  \big\}, \nonumber
\end{align}
where
\begin{align*}
\|f\|_{\dot{B}^s_{p,q}} := \Big (\sum_{j \in \mathbb{Z}} 2^{jsq} \|\dot{\Delta}_j f\|^q_{L^p}   \Big)^{\frac{1}{q}}.
\end{align*}
As we shall work with time-dependent functions valued in Besov spaces, we introduce the
norms:
$$\|f(t,\cdot)\|_{L^{q}_{T}(\dot{B}_{p,r}^{s})}:=\normo{\|f(t,\cdot)\|_{\dot{B}_{p,r}^{s}}
}_{L^{q}_{T}}.
$$

As highlighted in \cite{BaChDa11}, when applying parabolic estimates in Besov spaces, it seems natural to compute the Besov norm by first taking the time-Lebesgue norm before performing the summation. This approach leads to the introduction of the following quantities:
$$\|f(t,\cdot)\|_{\tilde{L}^{q}_{T}(\dot{B}_{p,r}^{s})}:=\| 2^{js}\|\dot{\Delta}_{j}f(t,\cdot)\|_{L_{T}^{q}(L^{p})}\|_{\ell^{r}(\mathbb{Z})}.
$$
 It's worth emphasizing that, due to the Minkowski inequality if $r \leq q$, then
\begin{equation}\label{ebedding}
\begin{aligned}
\|f(t,\cdot)\|_{L^{q}_{T}(\dot{B}_{p,r}^{s})}\leq \|f(t,\cdot)\|_{\tilde{L}^{q}_{T}(\dot{B}_{p,r}^{s})},
\end{aligned}
\end{equation}
with equality if and only if $q = r$. Conversely, the reverse inequality holds if $r \geq q$.
The corresponding truncated semi-norms are defined as follows for some $m_{0}\in\mathbb{Z}$:
$$\|f\|_{\dot{B}_{p,1}^{\sigma}}^{P_{\leq m_0}}:=\sum_{k\leq m_{0}}2^{k\sigma}\|\dot{\Delta}_{k}f\|_{L^{p}},\,\|f\|_{\dot{B}_{p,1}^{\sigma}}^{P_{> m_0}}:=
\sum_{k> m_{0}}2^{k\sigma}\|\dot{\Delta}_{k}f\|_{L^{p}}.
$$

In this paper, we use the frequency envelope method  \cite{Tao04} to show the continuity of the solution map. To this aim, we first introduce the following definition:
\begin{definition}\label{deft}
Let $\delta_{0}>0$. An acceptable frequency weight $\omega=\{\omega_{i}\}$ is defined as a sequence satisfying $\omega_{i}=1 $ for $ i\leq 0 $ and $ 1\leq \omega_{i}\leq \omega_{i+1}\leq 2^{\delta_{0}}\omega_{i}$ for $i>0$. We denote it by $\omega\in AF(\delta_{0})$.
\end{definition}
With an acceptable frequency weight $\omega$, we slightly modulate the homogeneous Besov spaces in the following way:
for $s\in \mathbb{R}$ and
 $1 \leq p, r \leq \infty$, we define $\dot{B}_{p,r}^{s}(\omega)$ with the norm
$$
  \|f\|_{\dot{B}_{p,r}^{s}(\omega)}
  :=(\sum_{i=-\infty}^{\infty}\omega_{i}^{r}2^{isr}\|\dot{\Delta}_{i}f\|_{L^{p}}^{r})^{1/r}.
$$
Note that $\dot{B}_{p,r}^{s}(\omega)=\dot{B}_{p,r}^{s}$ when we choose $\omega_{i}=1$ for any $i\in \mathbb{Z}$. Simlarly, we can define time-dependent functions valued in weighted Besov spaces
$$\|f(t,\cdot)\|_{L^{q}_{T}(\dot{B}_{p,r}^{s}(\omega))}:
=\|\|f(t,\cdot)\|_{\dot{B}_{p,r}^{s}(\omega)}
\|_{L^{q}_{T}}
$$
and
$$\|f(t,\cdot)\|_{\tilde{L}^{q}_{T}(\dot{B}_{p,r}^{s}(\omega))}:=\| 2^{js}\omega_{j}\|\dot{\Delta}_{j}f(t,\cdot)\|_{L_{T}^{q}L^{p}}\|_{\ell^{r}(\mathbb{Z})}.
$$

The following Bernstein's inequalities describe how derivatives act on spectrally localized functions.
\begin{lemma}[see Lemma 2.1 in \cite{BaChDa11}] \label{bernstein} Let $ \mathcal{B}$ be a Ball and $\mathcal{C}$ be an annulus. There exists a constant $C>0$ such that for all $k\in \mathbb{N}\cup \{0\}$, any positive real number $\lambda$ and any function $f\in L^p$ with $1\leq p \leq q \leq \infty$, we have
\begin{align*}
&{\rm{supp}}\hat{f}\subset \lambda \mathcal{B}\;\Rightarrow\; \|D^{k}f\|_{L^q}\leq C^{k+1}\lambda^{k+(\frac{1}{p}-\frac{1}{q})}\|f\|_{L^p},  \\
&{\rm{supp}}\hat{f}\subset \lambda \mathcal{C}\;\Rightarrow\; C^{-k-1}\lambda^{k}\|f\|_{L^p} \leq \|D^{k}f\|_{L^p} \leq C^{k+1}\lambda^{k}\|f\|_{L^p}.
\end{align*}
\end{lemma}

We now prove a frequency envelope version of continuity properties about the paraproduct and remainder operators. The following lemma will be frequently used.

\begin{lemma} \label{paraproduct}
Let $s\in\mathbb{R}, 1\leq p, r\leq\infty$ and $\frac{1}{p}+\frac{1}{p'}=1$. Assume $\omega\in AF(\delta_0)$.  Then there exists a constant $C$ such that
\begin{itemize}

\item[(1)]
\begin{equation}
\begin{aligned} 
\|\dot{T}_{f}g\|_{\dot{B}_{p,r
}^{s+t}(\omega)} \leq C\|f\|_{\dot{B}_{p,\infty}^{t+\frac{d}{p}}}\|g\|_{\dot{B}_{p,r}^{s}(\omega)}, \, \, t<0  \no,
\end{aligned}
\end{equation}
\begin{equation}\label{Tfg1}
\begin{aligned} \|\dot{T}_{f}g\|_{\dot{B}_{p,r
}^{s+t}(\omega)} \leq C\|f\|_{\dot{B}_{p,1}^{t+
\frac{d}{p}}}\|g\|_{\dot{B}_{p,r}^{s}(\omega)}, \, \, t\leq0.
\end{aligned}
\end{equation}

\item[(2)]
\begin{align}
\|\dot{T}_{f}g\|_{\dot{B}_{p,r
}^{s+t}(\omega)}
\leq C\|g\|_{\dot{B}_{p,r}^{s}} \|f\|_{\dot{B}_{p,\infty}^{t+\frac{d}{p}}(\omega)}, \,  \, t+\delta_{0}<0, \no
\end{align}
\begin{align}
\|\dot{T}_{f}g\|_{\dot{B}_{p,r
}^{s+t}(\omega)}
\leq C\|g\|_{\dot{B}_{p,\infty}^{s}} \|f\|_{\dot{B}_{p,r}^{t+\frac{d}{p}}(\omega)}, \,  \, t+\delta_{0}<0, \no
\end{align}
\begin{align}\label{Tfg4}
\|\dot{T}_{f}g\|_{\dot{B}_{p,r
}^{s+t}(\omega)}
\leq C\|g\|_{\dot{B}_{p,r}^{s}} \|f\|_{\dot{B}_{p,1}^{t+\frac{d}{p}}(\omega)}, \, \, t+\delta_{0}\leq0.
\end{align}
\item[(3)] if $t_{1}+t_{2}>-\min\{\frac{d}{p},\frac{d}{p'}\}$, then
\begin{equation}
\begin{aligned} \|\dot{R}(f,g)\|_{\dot{B}_{p,r
}^{t_{1}+t_{2}}(\omega)} \leq C\|f\|_{\dot{B}_{p,r}^{t_{1}+\frac{d}{p}}}\|g\|_{\dot{B}_{p,\infty}^{t_{2}}(\omega)},
\end{aligned}
\end{equation}
\begin{equation*}
\begin{aligned}
\|\dot{R}(f,g)\|_{\dot{B}_{p,r
}^{t_{1}+t_{2}}(\omega)}\leq C\|f\|_{\dot{B}_{p,\infty}^{t_{1}+\frac{d}{p}}}\|g\|_{\dot{B}_{p,r}^{t_{2}}(\omega)}.
\end{aligned}
\end{equation*}
\item[(4)] if $s_{1}\leq \frac{d}{p},s_{2}\leq \frac{d}{p}$ and $s_{1}+s_{2}> d\max\{0,\frac{2}{p}-1\}$, then we have
\begin{equation}
\begin{aligned} \|fg\|_{\dot{B}_{p,1
}^{s_{1}+s_{2}-\frac{d}{p}}(\omega)} \leq C\|f\|_{\dot{B}_{p,1}^{s_{1}}(\omega)}\|g\|_{\dot{B}_{p,1}^{s_{2}}}
+C\|f\|_{\dot{B}_{p,1}^{s_{1}}}\|g\|_{\dot{B}_{p,1}^{s_{2}}(\omega)}.
\end{aligned}
\end{equation}
\item[(5)] if $s_{1}\leq \frac{d}{p},s_{2}\leq \frac{d}{p}-\delta_{0}$ and $s_{1}+s_{2}>d\max\{0,\frac{2}{p}-1\}$, then we have
\begin{equation}
\begin{aligned} \|fg\|_{\dot{B}_{p,1
}^{s_{1}+s_{2}-\frac{d}{p}}(\omega)} \leq C\|f\|_{\dot{B}_{p,1}^{s_{1}}}\|g\|_{\dot{B}_{p,1}^{s_{2}}(\omega)}.
\end{aligned}
\end{equation}
\end{itemize}
\end{lemma}
\begin{proof}
(1) When $t>0$,
due to frequency interaction \eqref{TRfg}, \eqref{Sfg},  Lemma~\ref{bernstein} and Definition~\ref{deft}, one can obtain
\begin{align}
2^{m(s+t)}\omega_{m} \|\dot{\Delta}_m T_fg \|_{L^{p}}
 & \lesssim \sum_{|j-m| \leq 4} \frac{\omega_{m}}{\omega_{j}}2^{(m-j)(s+t)}2^{js}\omega_{j}\|\dot{\Delta}_j g\|_{L^{p}}\sum_{\ell \leq j-2}2^{(j-\ell) t}2^{\ell t}\|
 \dot{\Delta}_{\ell} f\|_{L^{\infty}}  \no \\
 & \lesssim \sum_{|j-m| \leq 4} 2^{js}2^{(m-j)(s+t)}\omega_{j}\|\dot{\Delta}_j g\|_{L^{p}}\sum_{\ell \leq j-2}2^{(j-\ell)t}2^{\ell (t+\frac{d}{p})}\|
 \dot{\Delta}_{\ell} f\|_{L^{p}}  \no \\
& \lesssim\|f\|_{\dot{B}_{p,\infty}^{t+\frac{d}{p}}}\sum_{|j-m| \leq 4} 2^{(m-j)s}2^{js}\omega_{j}\|\dot{\Delta}_j g\|_{L^{p}}.\no
\end{align}
Similarly, for $t\leq0$, we have
\begin{align}
2^{m(s+t)}\omega_{m} \|\dot{\Delta}_m T_fg \|_{L^{p}}
& \lesssim\|f\|_{\dot{B}_{p,1}^{t+\frac{d}{p}}}\sum_{|j-m| \leq 4} 2^{(m-j)s}2^{js}\omega_{j}\|\dot{\Delta}_j g\|_{L^{p}}.\no
\end{align}
Applying Young's inequality for the above series ensures that
\begin{equation*}
\begin{aligned} \|\dot{T}_{f}g\|_{\dot{B}_{p,r
}^{s+t}(\omega)} \leq C\|f\|_{\dot{B}_{p,\infty}^{t+\frac{d}{p}}}\|g\|_{\dot{B}_{p,r}^{s}(\omega)}, \, \, t<0
\end{aligned}
\end{equation*}
and
\begin{equation*}
\begin{aligned} \|\dot{T}_{f}g\|_{\dot{B}_{p,r
}^{s+t}(\omega)} \leq C\|f\|_{\dot{B}_{p,1}^{t+\frac{d}{p}}}\|g\|_{\dot{B}_{p,r}^{s}(\omega)}, \, \, t\leq0.
\end{aligned}
\end{equation*}
(2) When $t+\delta_{0}<0$, by \eqref{TRfg}, \eqref{Sfg} and
the H\"{o}lder inequality, we observe that
\begin{align*}
2^{m(s+t)}\omega_{m} \|\Delta_m \dot{T}_fg \|_{L^{p}}
 & \lesssim \sum_{|j-m| \leq 4} \frac{2^{(m-j)(s+t)}\omega_{m}}{\omega_{j}}\omega_{j}2^{jt}\|\dot{S}_{j-1} f\|_{L^{\infty}}2^{js}\|\dot{\Delta}_j g\|_{L^{p}},
\end{align*}
from which, by the Young inequality and  the H\"{o}lder inequality, we have
\begin{align*}
\|\dot{T}_{f}g\|_{\dot{B}_{p,r
}^{s+t}(\omega)}
& \lesssim \|g\|_{\dot{B}_{p,r}^{s}} \sup_{j\in\mathbb{ Z}}\{\sum_{k\leq j-2}\frac{\omega_{j}}{\omega_{k}}2^{(j-k)t}\omega_{k}2^{kt}\|\dot{\Delta}_{k} f\|_{L^{\infty}}\}  \no \\
& \lesssim \|g\|_{\dot{B}_{p,r}^{s}} \sup_{j\in\mathbb{ Z}}\{\sum_{k\leq j-2}2^{\delta_{0}(j-k)}2^{(j-k)t}\omega_{k}2^{k(t+\frac{d}{p})}\|\dot{\Delta}_{k} f\|_{L^{p}}\}  \no \\
& \lesssim \|g\|_{\dot{B}_{p,r}^{s}} \|f\|_{\dot{B}_{p,\infty}^{t+\frac{d}{p}}(\omega)}\no \\
\end{align*}
and
\begin{align*}
\|\dot{T}_{f}g\|_{\dot{B}_{p,r
}^{s+t}(\omega)}
& \lesssim \|g\|_{\dot{B}_{p,\infty}^{s}} \|\{\sum_{k\leq j-2}\frac{\omega_{j}}{\omega_{k}}2^{(j-k)t}\omega_{k}2^{kt}\|\dot{\Delta}_{k} f\|_{L^{\infty}}\}_{j} \|_{\ell^{r}(\mathbb{Z})} \no \\
& \lesssim \|g\|_{\dot{B}_{p,\infty}^{s}} \|\{\sum_{k\leq j-2}2^{\delta_{0}(j-k)}2^{(j-k)t}\omega_{k}2^{k(t+\frac{d}{p})}\|\dot{\Delta}_{k} f\|_{L^{p}}\}_{j} \|_{\ell^{r}(\mathbb{Z})}  \no \\
& \lesssim \|g\|_{\dot{B}_{p,\infty}^{s}} \|f\|_{\dot{B}_{p,r}^{t+\frac{d}{p}}(\omega)}. \no \\
\end{align*}
Similarly, when $t+\delta_{0}\leq0$, we get
\begin{align*}
\|\dot{T}_{f}g\|_{\dot{B}_{p,r
}^{s+t}(\omega)}
& \lesssim \|g\|_{\dot{B}_{p,r}^{s}} \sup_{j\in\mathbb{ Z}}(\sum_{k\leq j-2}2^{\delta(j-k)}2^{(j-k)t}\omega_{k}2^{k(t+\frac{d}{p})}\|\dot{\Delta}_{k} f\|_{L^{p}})  \no \\
& \lesssim \|g\|_{\dot{B}_{p,r}^{s}} \|f\|_{\dot{B}_{p,1}^{t+\frac{d}{p}}(\omega)}. \no \\
\end{align*}
(3) Notice that $1\leq p_{1}, p_{2}\leq\infty, \frac{1}{p}=\frac{1}{p_{1}}+\frac{1}{p_{2}}$ and
 $$\dot{\Delta}_{m}(\dot{R}(f,g))=\sum_{j\geq m-3} \dot{\Delta}_{m}(\dot{\Delta}_j f \tilde{\dot{\Delta}}_j g),$$
by the Hölder inequality, it follows that
\begin{align}
2^{m(s_{1}+s_{2})}\omega_{m} \|\dot{\Delta}_m \dot{R}(f,g) \|_{L^{p}}
 & \lesssim \sum_{m\leq j+3}2^{(m-j)(s_{1}+s_{2})} \frac{\omega_{m}}{\omega_{j+3}}\frac{\omega_{j+3}}{\omega_{j}}\omega_{j}2^{js_{1}}\|\dot{\Delta}_{j} f\|_{L^{p_{1}}}2^{js_{2}}\|\dot{\tilde{\Delta}}_j g\|_{L^{p_{2}}}  \no \\
  & \lesssim \sum_{m\leq j+3}2^{(m-j)(s_{1}+s_{2})} \omega_{j}2^{js_{1}}\|\dot{\Delta}_{j} f\|_{L^{p_{1}}}2^{js_{2}}\|\tilde{\dot{\Delta}}_j g\|_{L^{p_{2}}}.  \no
\end{align}
When $s_{1}+s_{2}>0$, applying Young's inequality for series yields
\begin{align}\label{rfg}
\|\dot{R}(f,g) \|_{_{\dot{B}_{p,r
}^{s_{1}+s_{2}}(\omega)}}
  & \lesssim \|g\|_{\dot{B}_{p_{2},r
}^{s_{2}}} \|f\|_{\dot{B}_{p_{1},\infty
}^{s_{1}}(\omega)} \sum_{m\leq3}2^{m(s_{1}+s_{2})} \lesssim \|g\|_{\dot{B}_{p_{2},r
}^{s_{2}}} \|f\|_{\dot{B}_{p_{1},\infty
}^{s_{1}}(\omega)}.
\end{align}
 When $p\geq2$, noting that $t_{1}+t_{2}+\frac{d}{p}>0$ and using \eqref{rfg} and Lemma~\ref{bernstein}, we obtain
 \begin{equation}\label{Rfg3}
\begin{aligned}
\|\dot{R}(f,g)\|_{\dot{B}_{p,r
}^{t_{1}+t_{2}}(\omega)} \lesssim\|\dot{R}(f,g)\|_{\dot{B}_{p/2,r
}^{t_{1}+t_{2}+\frac{d}{p}}(\omega)}\lesssim C\|f\|_{\dot{B}_{p,r}^{t_{1}+\frac{d}{p}}}\|g\|_{\dot{B}_{p,\infty}^{t_{2}}(\omega)}.
\end{aligned}
\end{equation}
For $1\leq p<2$, when $t_{1}+t_{2}+\frac{d}{p'}>0$, by \eqref{rfg} and Lemma~\ref{bernstein}, one has
 \begin{equation}\label{Rfg4}
\begin{aligned}
&\|\dot{R}(f,g)\|_{\dot{B}_{p,r
}^{t_{1}+t_{2}}(\omega)} \lesssim\|\dot{R}(f,g)\|_{\dot{B}_{1,r
}^{t_{1}+t_{2}+\frac{d}{p'}}(\omega)}\\
&\   \    \   \     \lesssim \|f\|_{\dot{B}_{p',r}^{t_{1}+\frac{d}{p'}}}\|g\|_{\dot{B}_{p,\infty}^{t_{2}}(\omega)}
\lesssim\|f\|_{\dot{B}_{p,r}^{t_{1}+\frac{d}{p}}}\|g\|_{\dot{B}_{p,\infty}^{t_{2}}(\omega)}.
\end{aligned}
\end{equation}
Thus by \eqref{Rfg3} and \eqref{Rfg4}, when $t_{1}+t_{2}>-\min\{\frac{d}{p},\frac{d}{p'}\}$, we can obtain
\begin{equation*}
\begin{aligned}
\|\dot{R}(f,g)\|_{\dot{B}_{p,r
}^{t_{1}+t_{2}}(\omega)}\lesssim\|f\|_{\dot{B}_{p,r}^{t_{1}+\frac{d}{p}}}\|g\|_{\dot{B}_{p,\infty}^{t_{2}}(\omega)}.
\end{aligned}
\end{equation*}
Similarly, when $t_{1}+t_{2}>-\min\{\frac{d}{p},\frac{d}{p'}\}$, we can obtain
\begin{equation*}
\begin{aligned}
\|\dot{R}(f,g)\|_{\dot{B}_{p,r
}^{t_{1}+t_{2}}(\omega)}\lesssim\|f\|_{\dot{B}_{p,\infty}^{t_{1}+\frac{d}{p}}}\|g\|_{\dot{B}_{p,r}^{t_{2}}(\omega)}.
\end{aligned}
\end{equation*}
(4) By the results of (1), for $s_{1}-\frac{d}{p}\leq 0$, we have
$$
\|\dot{T}_{f}g\|_{\dot{B}_{p,1}^{s_{1}+s_{2}-\frac{d}{p}}(\omega)}\lesssim\|f\|_{\dot{B}_{p,1}^{s_{1}}}\|g\|_{\dot{B}_{p,1}^{s_{2}}(\omega)},$$ and for $s_{2}-\frac{d}{p}\leq 0$, one has
$$
\|\dot{T}_{g}f\|_{\dot{B}_{p,1}^{s_{1}+s_{2}-\frac{d}{p}}(\omega)}\lesssim\|f\|_{\dot{B}_{p,1}^{s_{1}}(\omega)}\|g\|_{\dot{B}_{p,1}^{s_{2}}}.
$$
When $s_{1}+s_{2}> d\max\{0,\frac{2}{p}-1\}$, using the conclusions of (3), we have
$$\|\dot{R}(f,g)\|_{\dot{B}_{p,1}^{s_{1}+s_{2}-\frac{d}{p}}(\omega)}\lesssim\|f\|_{\dot{B}_{p,1}^{s_{1}}(\omega)}\|g\|_{\dot{B}_{p,1}^{s_{2}}}.$$
Therefore, by \eqref{fg}, we have the desired results.

(5) Similarly, by the results of (1) we have
$$\|\dot{T}_{f}g\|_{\dot{B}_{p,1}^{s_{1}+s_{2}-\frac{d}{p}}(\omega)}
\lesssim\|f\|_{\dot{B}_{p,1}^{s_{1}}}\|g\|_{\dot{B}_{p,1}^{s_{2}}(\omega)}
$$ for $s_{1}-\frac{d}{p}\leq 0$.
When $s_{2}-\frac{d}{p}\leq -\delta_{0}$, according to (2), one may write
$$
\|\dot{T}_{g}f\|_{\dot{B}_{p,1}^{s_{1}+s_{2}-\frac{d}{p}}(\omega)}
\lesssim\|f\|_{\dot{B}_{p,1}^{s_{1}}}\|g\|_{\dot{B}_{p,1}^{s_{2}}(\omega)}.
$$
For $s_{1}+s_{2}> d\max\{0,\frac{2}{p}-1\}$, Using the result of (3), one has
$$
\|\dot{R}(f,g)\|_{\dot{B}_{p,1}^{s_{1}+s_{2}-\frac{d}{p}}(\omega)}\lesssim\|f\|_{\dot{B}_{p,1}^{s_{1}}}\|g\|_{\dot{B}_{p,1}^{s_{2}}(\omega)}.
$$
At last, by \eqref{fg}, we have the desired results.
\end{proof}
 \begin{rem} \label{remark}
 From (4) of Lemma~\ref{paraproduct}, for some constant $C>0$, whenever
$1\leq p<\infty$ we deduce that
\begin{equation}\label{bfg}
\begin{aligned} \|fg\|_{\dot{B}_{p,1}^{\frac{d}{p}}(\omega)}\leq C\|f\|_{\dot{B}_{p,1}^{\frac{d}{p}}}\|g\|_{\dot{B}_{p,1}^{\frac{d}{p}}(\omega)}+C\|f\|_{\dot{B}_{p,1}^{\frac{d}{p}}(\omega)}\|g\|_{\dot{B}_{p,1}^{\frac{d}{p}}}.\end{aligned}
\end{equation}
From (5) of Lemma~\ref{paraproduct}, when $d\geq2, 1\leq p<2d, 0<\delta_{0}\leq1$, there exists some constant $C>0$ such that
\begin{equation}\label{Cfg}
\begin{aligned}\|fg\|_{\dot{B}_{p,1}^{-1+\frac{d}{p}}(\omega)}\leq C\|f\|_{\dot{B}_{p,1}^{\frac{d}{p}}}\|g\|_{\dot{B}_{p,1}^{-1+\frac{d}{p}}(\omega)}.\end{aligned}
\end{equation}
\end{rem}
We  consider the following transport system:
\begin{align}\label{eq:TDep}
\begin{cases}
\p_t a+v\cdot \nabla a+\lambda a=f,\\
a(0)=a_0,
\end{cases}
\end{align}
where $a_0=a_0(x)$ represents the initial data, $f=f(t,x)$ is the source term, $\lambda\geq 0$ is the damping coefficient, and $v=v(t,x)$ is the time-dependent transport field. We will provide estimates in Besov spaces for the transport equation with a frequency envelope version.

\begin{lemma}\label{transport}
Let $1\leq p,r\leq \infty$ and
\begin{align*}
-d \min(\frac{1}{p},1-\frac{1}{p}) <s<1+\frac{d}{p}-\delta_0.
\end{align*}
Assume that $\omega\in AF(\delta_0), a_{0}\in\dot{B}_{p,r}^{s}(\omega)$ and
$f\in\tilde{L}_{t}^{1}(\dot{B}_{p,r}^{s}(\omega))$, then there exists a constant $C>0$ such that for any smooth solution $a$ of \eqref{eq:TDep} and $t\geq 0$ we have
\begin{equation*}
\begin{aligned}
\|a\|_{\tilde{L}_{t}^{\infty}(\dot{B}_{p,r}^{s}(\omega))}+
\lambda\|a\|_{\tilde{L}_{t}^{1}(\dot{B}_{p,r}^{s}(\omega))}
\leq Ce^{CV(t)}(\|a_{0}\|_{\dot{B}_{p,r}^{s}(\omega)}
+\|f\|_{\tilde{L}_{t}^{1}(\dot{B}_{p,r}^{s}(\omega))})
\end{aligned}
\end{equation*}
with
\begin{equation*}
\begin{aligned}
V(t)=\int_{0}^{t}\|\nabla v(\tau)\|_{\dot{B}^{\frac{d}{p}}_{p,\infty}\cap L^\infty}d\tau.
\end{aligned}
\end{equation*}
\end{lemma}
\begin{proof}
   Applying the frequency localization operator $\dot{\Delta}_j $ to the transport equation \eqref{eq:TDep}, one can find
\begin{align*}
\partial_t \dot{\Delta}_j a+ v\cdot \nabla \dot{\Delta}_j a+ \lambda\dot{\Delta}_j a=\dot{\Delta}_j f+ [v\cdot \nabla,\dot{\Delta}_j] a.  \no
\end{align*}
We first claim that the commutator satisfies
 \begin{equation}\label{CR}
\begin{aligned}
&2^{js}\omega_{j}\|[v\cdot \nabla,\dot{\Delta}_j] a\|_{L^{p}} \lesssim c_{j}(t)\|\nabla v(t)\|_{\dot{B}^{\frac{d}{p}}_{p,\infty}\cap L^\infty}\|a(t)\|_{\dot{B}_{p,r}^{s}(\omega)}
\end{aligned}
\end{equation} with $\|c_{j}(t)\|_{\ell^{r}}=1$.
Indeed, using Bony's decomposition, we infer that
\begin{equation}\label{Rfg1}
\begin{aligned}
&[v\cdot \nabla,\dot{\Delta}_j] a=[\dot{T}_{v^{k}},\dot{\Delta}_j]\partial_{k}a+(T_{\partial_{k}\dot{\Delta}_ja}v^{k}
+\dot{R}(\partial_{k}\dot{\Delta}_ja,v^{k}))-\dot{\Delta}_{j}
\dot{T}_{\partial_{k}a}v^{k}-\dot{\Delta}_{j}\dot{R}(
\partial_{k}a,v^{k}).
\end{aligned}
\end{equation}
Due to the spectral localization properties of the Littlewood-Paley decomposition, we first observe that
\begin{equation*}
\begin{aligned}
&[\dot{T}_{v^{k}},\dot{\Delta}_j]\partial_{k}a=\sum_{|j-j'|\leq 4}[\dot{S}_{j'-1}v^{k},\dot{\Delta}_j]\partial_{k}\dot{\Delta}_ja,
\end{aligned}
\end{equation*}
from which, using commutator estimates of Lemma 2.97 in \cite{BaChDa11}, page 110, one can  obtain
\begin{equation}\label{CT1}
\begin{aligned}
&2^{js}\omega_{j}\|[\dot{T}_{v^{k}},\dot{\Delta}_j]\partial_{k}a\|_{L^{p}}\lesssim \sum_{|j-j'|\leq 4}2^{js}\omega_{j}2^{-j}
\|\dot{S}_{j'-1}\nabla v\|_{L^{\infty}}
\|\partial_{k}\dot{\Delta}_{j'}a\|_{L^{p}}\\
&\lesssim
\|\nabla v\|_{L^{\infty}}
\sum_{|j-j'|\leq 4}2^{(j-j')s}2^{-j+j'}\frac{\omega_{j}}{\omega_{j'}}2^{j's}\omega_{j'}\|\dot{\Delta}_{j'}a\|_{L^{p}}\\
&\lesssim c_{j}(t)\| a\|_{\dot{B}_{p,r}^{s}(\omega)}\|\nabla v\|_{L^{\infty}}.\\
\end{aligned}
\end{equation}
By (2) of Lemma~\ref{paraproduct}, when $s-1-\frac{d}{p}<-\delta_{0}$, we have
\begin{equation}\label{CT2}
\begin{aligned}
&2^{js}\omega_{j}\|\dot{\Delta}_{j}
\dot{T}_{\partial_{k}a}v^{k}\|_{L^{p}}=c_{j}(t)\|
\dot{T}_{\partial_{k}a}v^{k}\|_{\dot{B}_{p,r}^{s}(\omega)}\lesssim c_{j}(t)\|\nabla a\|_{\dot{B}_{p,r}^{s-1}(\omega)}\|v\|_{\dot{B}_{p,\infty}^{1+\frac{d}{p}}}\\
&\lesssim c_{j}(t)\| a\|_{\dot{B}_{p,r}^{s}(\omega)}\|\nabla v\|_{\dot{B}_{p,\infty}^{\frac{d}{p}}}.
\end{aligned}
\end{equation}
From (3) of Lemma~\ref{paraproduct} , for $s>-{\text{min}}\{\frac{d}{p},\frac{d}{p'}\}$, we get
\begin{equation}\label{CT3}
\begin{aligned}
&2^{js}\omega_{j}\|\dot{\Delta}_{j}\dot{R}(
\partial_{k}a,v^{k})\|_{L^{p}}=c_{j}(t)\|
\dot{R}(
\partial_{k}a,v^{k})\|_{\dot{B}_{p,r}^{s}(\omega)}\lesssim c_{j}(t)\|\nabla a\|_{\dot{B}_{p,r}^{s-1}(\omega)}\|v\|_{\dot{B}_{p,\infty}^{1+\frac{d}{p}}}\\
&\lesssim c_{j}(t)\| a\|_{\dot{B}_{p,r}^{s}(\omega)}\|\nabla v\|_{\dot{B}_{p,\infty}^{\frac{d}{p}}}.\\
\end{aligned}
\end{equation}
Notice that $$T_{\partial_{k}\dot{\Delta}_ja}v^{k}
+\dot{R}(\partial_{k}\dot{\Delta}_ja,v^{k})=\sum_{j\leq j'+2}S_{j'+2}\partial_{k}
\dot{\Delta}_ja\dot{\Delta}_{j'}v^{k}.$$
We get by Lemma~\ref{bernstein} that
 \begin{equation}\label{CT4}
\begin{aligned}
&2^{js}\omega_{j}\|T_{\partial_{k}\dot{\Delta}_ja}v^{k}
+\dot{R}(\partial_{k}\dot{\Delta}_ja,v^{k})\|_{L^{p}}=2^{js}\omega_{j}\sum_{j\leq j'+2}\|S_{j'+2}\partial_{k}
\dot{\Delta}_ja\|_{L^{p}}\|\dot{\Delta}_{j'}v^{k}\|_{L^{\infty}}\\
&\lesssim 2^{js}\omega_{j}\sum_{j\leq j'+2}2^{j-j'}\|
\dot{\Delta}_ja\|_{L^{p}}\|\nabla v^{k}\|_{L^{\infty}}\\
&\lesssim 2^{js}\omega_{j}\|
\dot{\Delta}_ja\|_{L^{p}}\|\nabla v\|_{L^{\infty}}\lesssim c_{j}(t)\| a\|_{\dot{B}_{p,r}^{s}(\omega)}\|\nabla v\|_{L^{\infty}}.\\
\end{aligned}
\end{equation}
Combining \eqref{Rfg1}-\eqref{CT4}, we can obtain \eqref{CR}.
Applying classical $L^{p}$ estimates \cite[Proposition 2.1]{D2001} for the transport equation
 implies
 \begin{equation}\label{Lp1}
\begin{aligned}
&\|\dot{\Delta}_{j}a(t)\|_{L^{p}} +\lambda2^{js}\omega_{j}\|\dot{\Delta}_{j}a\|_{L_{t}^{1}(L^{p})}\leq\|\dot{\Delta}_j a_{0}\|_{L^{p}}\\&+\int_0^t (\|\dot{\Delta}_{j}f\|_{L^{p}}+\frac{\|{\rm div}\, v\|_{L^{\infty}}}{p}\|\dot{\Delta}_{j} a\|_{L^{p}}+\|[v\cdot \nabla,\dot{\Delta}_{j}] a\|_{L^{p}}) d\tau.
\end{aligned}
\end{equation}
Multiplying the above terms by $2^{js}\omega_{j}$ and using
\eqref{CR}, taking the $\ell^{r}$ norm in \eqref{Lp1} imply
\begin{equation}\label{a-pri-bef}
\begin{aligned}
\|a\|_{\tilde{L}_{t}^{\infty}\dot{B}_{p,r}^{s}(\omega)}+
\lambda\|a\|_{\tilde{L}_{t}^{1}\dot{B}_{p,r}^{s}(\omega)}
\lesssim\|a_{0}\|_{\dot{B}_{p,r}^{s}(\omega)}+\|f\|_{\tilde{L}_{t}^{1}\dot{B}_{p,r}^{s}(\omega)}+
\int_{0}^{t}V'(\tau)\|a(\tau)\|_{\dot{B}_{p,r}^{s}(\omega)}d\tau.
\end{aligned}
\end{equation}
Then applying Gronwall's lemma gives the desired inequality for $a$.
\end{proof}
\begin{lemma} \label{function1}
Let $F: \mathbb{R}\rightarrow \mathbb{R}$ be
smooth with $F(0) = 0$. For all $1\leq p,r\leq\infty$ and $s>0$, $\omega\in AF(\delta_0)$, we have
$F(f)\in \dot{B}_{p,r}^{s}(\omega)$ for all $f\in \dot{B}_{p,r}^{s}(\omega)\cap L^{\infty}$, and
$$\|F(f)\|_{\dot{B}_{p,r}^{s}(\omega)}\leq C\|f\|_{\dot{B}_{p,r}^{s}(\omega)},
$$
with $C$ depending only on $\|f\|_{L^{\infty}}$, $F$, $s, p, r, d$ and $\delta_{0}$.
\end{lemma}
\begin{proof}
Since $\lim\limits_{j\rightarrow\infty}\|\dot{S}_{j}u-u\|_{L^{p}}=0$ and $F(0)=0$. From the mean value theorem, we introduce the telescopic series
\begin{equation}\label{Ff}
\begin{aligned}
F(f)=\sum_{j}(F(\dot{S}_{j+1}f)-F(\dot{S}_{j}f))
:=\sum_{j}m_{j}\dot{\Delta}_{j}f,
\end{aligned}
\end{equation}
with $m_{j}=\int_{0}^{1}F'(\dot{S}_{j}f+\theta\dot{\Delta}_{j}f)d\theta$. Let
$F_{j}:=m_{j}\dot{\Delta}_{j}f$. For $\gamma\in \mathbb{N}^{d}$, using the arguments of Lemma 2.63 in \cite{BaChDa11}, page 95, one can obtain $\|D^{\gamma}m_{j}\|_{L^{\infty}}\leq C(\gamma, F,\|f\|_{L^{\infty}})2^{j|\gamma|}$, from which and \eqref{Ff}, it follows that
\begin{equation*}
\begin{aligned}
\|\dot{\Delta}_{j}F(f)\|_{L^{p}}&\lesssim\sum_{j'\leq j}\|\dot{\Delta}_{j}F_{j'}\|_{L^{p}}+\sum_{j'>j}\|\dot{\Delta}_{j}F_{j'}\|_{L^{p}}\\
&\lesssim\sum_{j'\leq j}2^{-j|\alpha|}\|\dot{\Delta}_{j}D^{\alpha}F_{j'}\|_{L^{p}}+\sum_{j'>j}
\|\dot{\Delta}_{j}F_{j'}\|_{L^{p}}\\
&\lesssim\sum_{j'\leq j}2^{-j|\alpha|}\sum_{\beta\leq\alpha}\|D^{\alpha-\beta}m_{j'}\|_{L^{\infty}}\|D^{\beta}\dot{\Delta}_{j'}f\|_{L^{p}}+\sum_{j'>j}
\|\dot{\Delta}_{j'}f\|_{L^{p}}\|m_{j'}\|_{L^{\infty}}\\
&\lesssim\sum_{j'\leq j}2^{(j'-j)|\alpha|}\|\dot{\Delta}_{j'}f\|_{L^{p}}+\sum_{j'>j}
\|\dot{\Delta}_{j'}f\|_{L^{p}},\\
\end{aligned}
\end{equation*}
with $\alpha\in \mathbb{N}^{d}$ to be determined later.
Multiplying the above term by $2^{js}\omega_{j}$  implies
\begin{equation}\label{wj}
\begin{aligned}
2^{js}\omega_{j}&\|\dot{\Delta}_{j}F(f)\|_{L^{p}}\\
&\lesssim\sum_{j'\leq j}2^{(j'-j)(|\alpha|-s)}\frac{\omega_{j}}{\omega_{j'}}\omega_{j'}2^{j's}\|\dot{\Delta}_{j'}u\|_{L^{p}}  +\sum_{j'>j}
2^{(j-j')s}\frac{\omega_{j}}{\omega_{j'}}\omega_{j'}2^{j's}\|\dot{\Delta}_{j'}f\|_{L^{p}}\\
&\lesssim\sum_{j'\leq j}2^{(j'-j)(|\alpha|-s-\delta_{0})}\omega_{j'}2^{j's}\|\dot{\Delta}_{j'}u\|_{L^{p}} +\sum_{j'>j}
2^{(j-j')s}\omega_{j'}2^{j's}\|\dot{\Delta}_{j'}f\|_{L^{p}},\\
\end{aligned}
\end{equation}
 if we take $|\alpha|>0$ large enough such that $|\alpha|>s+\delta_{0}$ with $s>0$, summing up over $j$ for \eqref{wj}, from Young's inequality, we have $$\|F(f)\|_{\dot{B}_{p,r}^{s}(\omega)}\leq C\|f\|_{\dot{B}_{p,r}^{s}(\omega)}.
$$
We thus finish the proof of Lemma ~\ref{function1}.
\end{proof}
We prove a frequency envelope version of maximal regularity for the following heat equations:
\begin{equation} \label{heat}
\begin{cases}
\partial_t u - \mu\Delta u    =  f,  & \\
u(0,x)= u_0(x).&
\end{cases}
\end{equation}
The external source term $f=f(t,x)$ and the initial data $u_{0}=u_{0}(x)$ are given.
\begin{lemma}\label{heatm}
Let $T>0, s\in\mathbb{R}$ and $1\leq \rho_{1}, \rho,p, r\leq\infty$. Assume that  $\omega\in AF(\delta_0)$, $u_{0}\in\dot{B}_{p,r}^{s}(\omega)$ and $f\in \tilde{L}_{T}^{\rho}(\dot{B}_{p,r}^{s-2+\frac{2}{\rho}}(\omega))$.Then
\eqref{heat} has a uniqueness solution $u$ in $\tilde{L}_{T}^{1}(\dot{B}_{p,r}^{s+2}(\omega))\cap\tilde{L}_{T}^{\infty}
(\dot{B}_{p,r}^{s}(\omega))$ and there exists a constant $C$ depending only on $d$ such that for all $\rho_{1}\in[\rho,+\infty]$, we have
\begin{equation*}
\begin{aligned}
&\mu^{1/\rho_{1}}\|u\|_{\tilde{L}_{T}^{\rho_{1}}(\dot{B}_{p,r}^{s+\frac{2}{\rho_{1}}}(\omega))}
\leq
C(\|u_{0}\|_{\dot{B}_{p,r}^{s}(\omega)}
+\mu^{1/\rho-1}\|f\|_{\tilde{L}_{T}^{\rho}(\dot{B}_{p,r}^{s-2+\frac{2}{\rho}}(\omega))}).
\end{aligned}
\end{equation*}
\end{lemma}
\begin{proof}
The solution  $u$ of \eqref{heat} can be written as
 \begin{align}
  u(t)=e^{\mu t\Delta}u_{0}+\int_{0}^{t}e^{\mu (t-\tau)\Delta}f(\tau)d\tau,  \no
\end{align}
from which and  Lemma 2.4 in \cite{BaChDa11}, page 54, we  have for some $\kappa>0$,
\begin{equation}\label{qu}
\begin{aligned}
\|\dot{\Delta}_q u(t)\|_{L^{p}}\lesssim e^{-\kappa\mu 2^{2q}t}\|\dot{\Delta}_q u_{0}\|_{L^{p}}+\int_{0}^{t}e^{\mu\kappa2^{2q}(t-\tau)\Delta}\|\dot{\Delta}_q f(\tau)\|_{L^{p}}d\tau.
\end{aligned}
\end{equation}
Taking $L^{\rho_{1}}$ norm for \eqref{qu} with respect to $t$ gives
\begin{equation*}\label{2.30}
\begin{aligned}
\|\dot{\Delta}_q u(t)\|_{L_{T}^{\rho_{1}}(L^{p})}\lesssim (\frac{1-e^{-\kappa\mu T\rho_{1} 2^{2q}}}{\kappa\mu \rho_{1} 2^{2q}})^{\frac{1}{\rho_{1}}}\|\dot{\Delta}_q u_{0}\|_{L^{p}}+
(\frac{1-e^{-\kappa\mu T\rho_{2} 2^{2q}}}{\kappa\mu \rho_{2} 2^{2q}})^{\frac{1}{\rho_{2}}}
\|\dot{\Delta}_q f(\tau)\|_{L_{T}^{\rho}(L^{p})}
\end{aligned}
\end{equation*}
with $\frac{1}{\rho_{2}}=1+\frac{1}{\rho_{1}}-\frac{1}{\rho}$.
Multiplying the above terms by $2^{q(s+\frac{2}{\rho_{1}})}\omega_{q}$ and taking the $\ell^{r}(\mathbb{Z})$ norm, we conclude that
\begin{equation*}
\begin{aligned}
&\|u\|_{\tilde{L}_{T}^{\rho_{1}}(\dot{B}_{p,r}^{s+2/\rho_{1}}(\omega))}
\lesssim \left[\sum_{q\in \mathbb{Z}}\left(\frac{1-e^{-\kappa\mu T\rho_{1} 2^{2q}}}{\kappa\mu \rho_{1} }\right)^{\frac{r}{\rho_{1}}}2^{qsr}\omega_{q}^{r}\|\dot{\Delta}_q u_{0}\|_{L^{p}}^{r}\right]^{\frac{1}{r}}
\\
&+
\left[\sum_{q\in \Z}\left(\frac{1-e^{-\kappa\mu T\rho_{2} 2^{2q}}}{\kappa\mu \rho_{2} }\right)^{\frac{r}{\rho_{2}}}2^{q(s-2+\frac{2}{\rho})r}
\omega_{q}^{r}\|\dot{\Delta}_q f(\tau)\|_{L_{T}^{\rho}(L^{p})}^{r}\right]^{\frac{1}{r}}
\\
&
\lesssim\mu^{-\frac{1}{\rho_{1}}}\|u_{0}\|_{\dot{B}_{p,r}^{s}(\omega)}
+\mu^{\frac{1}{\rho}-1-\frac{1}{\rho_{1}}}\|f\|_{\tilde{L}_{T}^{\rho}(\dot{B}_{p,r}^{s-2+\frac{2}{\rho}}(\omega))},
\end{aligned}
\end{equation*}
which ensures that $u\in\tilde{L}_{T}^{1}(\dot{B}_{p,r}^{s+2}(\omega))\cap\tilde{L}_{T}^{\infty}
(\dot{B}_{p,r}^{s}(\omega))$ and yields the desired inequality.
\end{proof}
As an application of Lemma \ref{heatm}, we get similar estimates for the following Lamé system
\begin{equation} \label{lame}
\begin{cases}
\partial_t u - \mu\Delta u -(\lambda+\mu)\nabla  {\rm div}\, u =  f,  & \\
u(0,x)= u_0(x).&
\end{cases}
\end{equation}
\begin{corollary}\label{cor:heatm}
Let $T>0, \omega \in AF(\delta_0),  s\in\mathbb{R}, 1\leq p\leq\infty, \mu>0$ and $\lambda+2\mu>0$. Assume that $u_{0}\in\dot{B}_{p,1}^{s}(\omega)$ and $f\in \tilde{L}_{T}^{1}(\dot{B}_{p,1}^{s}(\omega))$. Then
\eqref{lame} has a unique solution $u$ in $\tilde{L}_{T}^{1}(\dot{B}_{p, 1}^{s+2}(\omega))\cap\tilde{L}_{T}^{\infty}
(\dot{B}_{p,1}^{s}(\omega))$ and there exists a constant $C$ depending only on $d$ such that
\begin{equation*}
\begin{aligned}
&\|u\|_{\tilde{L}_{T}^{\infty}(\dot{B}_{p,1}^{s}(\omega))}+\min\{\mu, 2\mu+\lambda\}\|u\|_{\tilde{L}_{T}^{1}(\dot{B}_{p,1}^{s+2}(\omega))}
\leq
C(\|u_{0}\|_{\dot{B}_{p,1}^{s}(\omega)}
+ \|f\|_{\tilde{L}_{T}^{1}(\dot{B}_{p,1}^{s}(\omega))}).
\end{aligned}
\end{equation*}
\end{corollary}
\begin{proof}
Let $\mathbb{P}:=I-\nabla(-\Delta)^{-1}{\rm div}$ and $\mathbb{Q}:=\nabla(-\Delta)^{-1}{\rm div}$ be the orthogonal projectors over divergence-free and potential vector fields. Applying $\mathbb{P}$ and $\mathbb{Q}$ to \eqref{lame}  gives
 \begin{equation} \label{plame}
\partial_t \mathbb{P}u - \mu\Delta \mathbb{P}u=\mathbb{P}  f, \  \ \ \,
\partial_t \mathbb{Q}u - (2\mu+\lambda)\Delta \mathbb{P}u=\mathbb{Q}  f.
\end{equation}
Using Lemma \ref{heatm} for \eqref{plame}  yields that
 \begin{equation*}
\begin{aligned}
&\|\mathbb{P}u\|_{\tilde{L}_{T}^{\infty}(\dot{B}_{p,1}^{s}(\omega))}+\mu
\|\mathbb{P}u\|_{\tilde{L}_{T}^{1}(\dot{B}_{p,1}^{s+2}(\omega))}
\leq
C(\|\mathbb{P}u_{0}\|_{\dot{B}_{p,r}^{s}(\omega)}
+\|\mathbb{P}f\|_{\tilde{L}_{T}^{1}
(\dot{B}_{p,1}^{s}(\omega))})
\end{aligned}
\end{equation*}
and
\begin{equation*}
\begin{aligned}
&\|\mathbb{Q}u\|_{\tilde{L}_{T}^{\infty}(\dot{B}_{p,1}^{s}(\omega))}+ (2\mu+\lambda)\|\mathbb{Q}u\|_{\tilde{L}_{T}^{1}(\dot{B}_{p,1}^{s+2}(\omega))}\leq
C(\|\mathbb{Q}u_{0}\|_{\dot{B}_{p,1}^{s}(\omega)}
+\|\mathbb{Q}f\|_{\tilde{L}_{T}^{1}(\dot{B}_{p,1}^{s}(\omega))}).
\end{aligned}
\end{equation*}
As $\mathbb{P}$ and $\mathbb{Q}$ are $0$ order multiplies, we conclude
the desired result.
\end{proof}
\begin{rem} \label{remark1}
In particular, when $\delta_{0}=0$ and $\omega_{i}=1$ for all $i\in \mathbb{Z}$, all the conclusions of Remark \ref{remark}, Lemma \ref{paraproduct}, Lemma \ref{transport}, Lemma \ref{function1}, corollary \ref{cor:heatm} also hold true. From (5) of Lemma~\ref{paraproduct}, when $d>2, 1\leq p<d$, there exists some constant $C>0$ such that
$$
\|fg\|_{\dot{B}_{p,1}^{-2+\frac{d}{p}}}\leq C\|f\|_{\dot{B}_{p,1}^{\frac{d}{p}}}\|g\|_{\dot{B}_{p,1}^{-2+\frac{d}{p}}},\quad
\|fg\|_{\dot{B}_{p,1}^{-2+\frac{d}{p}}}\leq C\|f\|_{\dot{B}_{p,1}^{-1+\frac{d}{p}}}\|g\|_{\dot{B}_{p,1}^{-1
+\frac{d}{p}}}.
$$
\end{rem}

\section{Frequency envelope method: $1\leq p<2d$}\label{prooflwp}

In this section, we shall prove Theorem \ref{lwp}. We fix $1\leq p<2d$ throughout this section. The local existence and uniqueness results have been established in \cite{D2014, RH2015}. The main task is to show the continuity of the solution map.

We first recall the existence and uniqueness results obtained in \cite{D2014, RH2015}.
Suppose that the initial data $(a_0,u_0) \in \dot{B}_{p,1}^{\frac{d}{p}}\times\dot{B}_{p,1}^{-1+\frac{d}{p}}$ and for a sufficiently small constant $c$, $\|a_{0}\|_{\dot{B}_{p,1}^{\frac{d}{p}}}\leq c$.  Then there exists some time $T$ satisfying
\begin{align*}
\sum_{j\in \mathbb{Z}}(1-e^{-C_02^{2j}T})
2^{j(-1+\frac{d}{p})}\|\dot{\Delta}_j u_0\|_{L^p}
\leq \frac{c}{1+\|u_0\|_{\dot{B}_{p,1}^{\frac{d}{p}}}},
\end{align*}
such that the system \eqref{eulercauchy} admits a unique  solution $(a,u)\in Z_p(T)$. Moreover, $(a,u)$ satisfies
\begin{align*}
\|a\|_{L_{T}^{\infty}(\dot{B}_{p,1}^{\frac{d}{p}})}\lesssim c,\,
\|u\|_{L^1_T(B^{1+\frac{d}{p}}_{p,1})}\lesssim \log(1+c),\,
\|u\|_{E_p(T)}\lesssim \|u_0\|_{\dot{B}_{p,1}^{\frac{d}{p}}}+T,
\end{align*}
where
$$E_p(T):=C([0,T];\dot{B}_{p,1}^{-1+\frac{d}{p}})\cap L^{1}([0,T]; \dot{B}_{p,1}^{1+\frac{d}{p}}).$$
Then we conclude that for a fixed
$(a_0,u_0)$, we can find a neighborhood $U$ and a common time $T(U)$ within which the constructed solution has uniform estimates. That is, there exists
a $\eta>0$ and a uniform time and  uniform bounds $C_1,C_2,C_3$, such that for any $(\tilde{a_0},\tilde{u_0})$ with $\|(\tilde{a_0},\tilde{u_0})-(a_0,u_0)\|_{\dot{B}_{p,1}^{\frac{d}{p}}\times \dot{B}_{p,1}^{-1+\frac{d}{p}}}\leq \eta$, the solution $(\tilde{a},\tilde{u})\in C([0,T];\dot{B}^{-1+\frac{d}{p}}_{p,1})\times E_p(T)$ exists, is unique and belongs to
$\mathbb{E}_{uni}$, where
\begin{align}
\label{uniform000}
\mathbb{E}_{uni}:=\{(f,g): \|f\|_{L_{T}^{\infty}(\dot{B}_{p,1}^{\frac{d}{p}})}\leq C_1c,\,
\|g\|_{L^1_T(B^{1+\frac{d}{p}}_{p,1})}\leq C_2\log(1+c),\,
\|g\|_{E_p(T)}\leq C_3\}.
\end{align}

Now we assume a sequence $(a_{0}^{n}, u_0^n)$
converges to $(a_{0},u_{0})$ in $\dot{B}_{p,1}^{\frac{d}{p}}\times \dot{B}_{p,1}^{-1+\frac{d}{p}}$. We claim that there exists a sequence $\{\omega_i\}$ of positive numbers which satisfies
\begin{equation}\label{weight}
\begin{aligned}
0<\omega_{i}\leq\omega_{i+1}\leq 2^{\frac{1}{2}}\omega_{i},\, \,\lim\limits_{i\rightarrow +\infty}\omega_i=+\infty,
\end{aligned}
\end{equation}
here we take $\delta_{0}=\frac{1}{2}$ in Definition~\ref{deft}, such that
\begin{align}\label{au-ini-omega}
\sup_{n}\big(\|a_{0}^{n}\|_{\dot{B}_{p,1}^{\frac{d}{p}}(\omega)}+
\|u_{0}^{n}\|_{\dot{B}_{p,1}^{-1+\frac{d}{p}}(\omega)}\big)<\infty.
\end{align}
This can be done by following the same lines as that of Lemma 4.1 in \cite{KoTz}. For the convenience of the readers, we give a sketch of the proof here. Define
$$A_{i}^{n}: =2^{j\frac{d}{p}}\|\Delta_{j}a_{0}^{n}\|_{L^{p}}+2^{j(-1+\frac{d}{p})}\|\Delta_{j}u_{0}^{n}\|_{L^{p}}.$$ Since $(a_{0}^{n}, u_0^n)$
 converges to $(a_{0}^{\infty},u_{0}^{\infty}):=(a_{0},u_{0})$ as $n\rightarrow\infty$ in $\dot{B}_{p,1}^{\frac{d}{p}}\times \dot{B}_{p,1}^{-1+\frac{d}{p}}$, which implies that
$(A_{i}^n)|_{i\in \mathbb{Z}}\rightarrow (A_i^\infty)|_{i\in \mathbb{Z}}$ in $l^1(\mathbb{Z})$. Then for all $k\in \mathbb{N}$, there exist a strictly monotone $\mathcal{N}_k$ such that $\sup\limits_{n}\sum\limits_{i=\mathcal{N}_k}^{+\infty}A_i^n<2^{-k}.$ If $i<\mathcal{N}_1$, we set $\omega_i=1$, whereas if $i\geq \mathcal{N}_1$, there exists a unique $k\in \mathbb{N}$ such that $\mathcal{N}_k\leq i<\mathcal{N}_{k+1}$, we set $\omega_i=2^{\frac{k}{2}}$. Then we can obtain that
\begin{align*}
\sum_{i=-\infty}^{+\infty}w_iA_i^n
=\sum_{i=-\infty}^{\mathcal{N}_1-1}A_i^n+\sum_{k=1}^{+\infty}\sum_{\mathcal{N}_k}^{\mathcal{N}_{k+1}}2^{\frac{k}{2}}A_i^n\leq
\sum_{i=-\infty}^{+\infty}A_i^n+\sum_{k=1}^{+\infty}2^{\frac{k}{2}}2^{-k},
\end{align*}
which thus implies that the uniform estimate \eqref{au-ini-omega} holds true.

The main task is to show $(a^n,u^n)$ converges to $(a,u)$ in $Z_p(T)$. We will prove this by three steps.

{\bf Step 1.} Uniform small tails estimate on frequency:
$\forall \varepsilon>0$, there exists $N>0$ such that
\begin{equation}\label{highfreN}
\begin{aligned}
\|(P_{>N}a^n, P_{>N}u^n))\|_{Z_p(T)}\leq \varepsilon, \quad \forall n.
\end{aligned}
\end{equation}
We will prove this using the following proposition.

\begin{prop}\label{au-omega-prior}
Assume $\omega\in AF(\delta_0)$.
Let $(a,u)$ be a solution of \eqref{eulercauchy} and $0\leq t\leq T$. Assume that
$\|a\|_{L^\infty_T(\dot{B}_{p,1}^{\frac{d}{p}})}\leq c_{0}$
for some sufficiently small $c_{0}$ and $(a_0,u_0) \in \dot{B}_{p,1}^{\frac{d}{p}}(\omega)\times\dot{B}_{p,1}^{\frac{d}{p}}(\omega)$. Then there holds
\begin{align*}
\|a\|_{\tilde{L}_{T}^{\infty}(\dot{B}_{p,1}^{\frac{d}{p}}(\omega))}&+\|u\|_{\tilde{L}_{T}^{\infty}(\dot{B}_{p,1}^{-1+\frac{d}{p}}(\omega))}+\|u\|_{L^1_T(\dot{B}_{p,1}^{1+\frac{d}{p}}(\omega))}\\ \leq &C\exp\big\{C\int_0^T(1+\|u\|_{\dot{B}_{p,1}^{1+\frac{d}{p}}})\mathrm{d}t\big\}
(\|a_{0}\|_{\dot{B}_{p,1}^{\frac{d}{p}}(\omega)}+\|u_0\|
_{\dot{B}_{p,1}^{-1+\frac{d}{p}}(\omega))}).
\end{align*}
\end{prop}

Let's first assume the above proposition holds true. Applying Proposition \ref{au-omega-prior} to the sequence $(a^n,u^n)$ and  $\omega$ defined in~\eqref{weight}, by the uniform bounds $(a^{n},u^{n})\in\mathbb{E}_{uni}$ and \eqref{au-ini-omega} we can get an improved a-priori estimate: there exists a uniform $C_4$, such that
\begin{align*}
\|a^n\|_{\tilde{L}_{T}^{\infty}(\dot{B}_{p,1}^{\frac{d}{p}}(\omega))}&+\|u^n\|_{\tilde{L}_{T}^{\infty}(\dot{B}_{p,1}^{-1+\frac{d}{p}}(\omega))}+\|u^n\|_{L^1_T(\dot{B}_{p,1}^{1+\frac{d}{p}}(\omega))}\leq C_4, \quad \forall \, n.
\end{align*}
$\forall\, \varepsilon>0$, since $\lim\limits_{i\rightarrow +\infty}\omega_i=+\infty$, there exists a  number $N>0$ such that $\omega_{N}\geq \frac{C_4}{\varepsilon}$. Then using the monotone property $w_i\leq w_{i+1}$ for $i\in \mathbb{Z}$, we have
\begin{equation}\label{highfresmall}
\begin{aligned}
\|a^n&\|_{\tilde{L}_{T}^{\infty}(\dot{B}_{p,1}^{\frac{d}{p}})}^{p>N}+\|u^n\|_{\tilde{L}_{T}^{\infty}(\dot{B}_{p,1}^{-1+\frac{d}{p}})}^{p>N}+\|u^n\|_{L^1_T(\dot{B}_{p,1}^{1+\frac{d}{p}})}^{p>N}
 \\ \leq &\frac{1}{\omega_{N}}(
\|a^n\|_{\tilde{L}_{T}^{\infty}(\dot{B}_{p,1}^{\frac{d}{p}}(\omega))}^{p>N}+\|u^n\|_{\tilde{L}_{T}^{\infty}(\dot{B}_{p,1}^{-1+\frac{d}{p}}(\omega))}^{p>N}+\|u^n\|_{L^1_T(\dot{B}_{p,1}^{1+\frac{d}{p}}(\omega))}^{p>N})
\leq \frac{C_4}{\omega_{N}} \leq \varepsilon, \quad \forall \, n.
\end{aligned}
\end{equation}
Hence taking advantage of embbedding \eqref{ebedding} for $q=\infty$, from \eqref{highfresmall}, we obtain the desired result \eqref{highfreN}.

\begin{proof}[Proof of Proposition \ref{au-omega-prior}]
We use the term (4) of Lemma \ref{paraproduct} and Remark \ref{remark} to obtain that
\begin{align*}
&\|(1+a) {\rm div}\,u\|_{\dot{B}_{p,1}^{\frac{d}{p}}(\omega)}\lesssim (1+\|a\|_{\dot{B}_{p,1}^{\frac{d}{p}}})\|{\rm div}\,u\|_{\dot{B}_{p,1}^{\frac{d}{p}}(\omega)}+\|a\|_{\dot{B}_{p,1}^{\frac{d}{p}}(\omega)}\|{\rm div}\,u\|_{\dot{B}_{p,1}^{\frac{d}{p}}}.
\end{align*}
From the estimate \eqref{a-pri-bef} in Lemma \ref{transport} with $s=\frac{d}{p}$, we deduce that
\begin{align}\label{a-omega}
\|a\|_{\tilde{L}_{t}^{\infty}(\dot{B}_{p,1}^{\frac{d}{p}}(\omega))}\leq \|a_{0}\|_{\dot{B}_{p,1}^{\frac{d}{p}}(\omega)}+C\int_0^t[&(1+\|a\|_{\dot{B}_{p,1}^{\frac{d}{p}}})\|  u\|_{\dot{B}_{p,1}^{1+\frac{d}{p}}(\omega)}+\| u\|_{\dot{B}_{p,1}^{1+\frac{d}{p}}}\|a\|_{\dot{B}_{p,1}^{\frac{d}{p}}(\omega)}]\mathrm{d}\tau.
\end{align}
We apply the term (5) of Lemma \ref{paraproduct} and Remark \ref{remark} to get that
$$\|u\cdot\nabla u \|_{\dot{B}_{p,1}^{-1+\frac{d}{p}}(\omega)}\lesssim \|\nabla u\|_{\dot{B}_{p,1}^{\frac{d}{p}}}\|u\|_{\dot{B}_{p,1}^{-1+\frac{d}{p}}(\omega)},$$\,
$$\|I(a)\mathcal{A}(u) \|_{\dot{B}_{p,1}^{-1+\frac{d}{p}}(\omega)}\lesssim \|I(a)\|_{\dot{B}_{p,1}^{\frac{d}{p}}}\|\mathcal{A}u\|_{\dot{B}_{p,1}^{-1+\frac{d}{p}}(\omega)}.
$$
Then by virtue of Lemma \ref{function1}, we have
\begin{align*}
&\|I(a)\mathcal{A}(u) \|_{\dot{B}_{p,1}^{-1+\frac{d}{p}}(\omega)}
\lesssim\|a\|_{\dot{B}_{p,1}^{\frac{d}{p}}}\|u\|_{\dot{B}_{p,1}^{1+\frac{d}{p}}(\omega)},\,
\|\nabla G(a)\|_{\dot{B}_{p,1}^{-1+\frac{d}{p}}(\omega)}\lesssim \|a\|_{\dot{B}_{p,1}^{\frac{d}{p}}(\omega)}.
\end{align*}
Together with taking $s=-1+\frac{d}{p}$ in corollary \ref{cor:heatm}, we get
\begin{align}\label{u-omega}
\|u\|_{\tilde{L}_{t}^{\infty}(\dot{B}_{p,1}^{-1+\frac{d}{p}}(\omega))}&+\|u\|_{L^1_t(\dot{B}_{p,1}^{1+\frac{d}{p}}(\omega))}\leq
C\|u_0\|_{\dot{B}_{p,1}^{-1+\frac{d}{p}}(\omega)}\no\\&
+C\int_0^t(\|u\|_{\dot{B}_{p,1}^{1+\frac{d}{p}}}\|  u\|_{\dot{B}_{p,1}^{-1+\frac{d}{p}}(\omega)}+\|a\|_{B_{p,1}^{\frac{d}{p}}}\| u\|_{\dot{B}_{p,1}^{1+\frac{d}{p}}(\omega)}+\|a\|_{\dot{B}_{p,1}^{\frac{d}{p}}(\omega)})\mathrm{d}\tau.
\end{align}
Recall that $\|a\|_{L^\infty_T(\dot{B}_{p,1}^{\frac{d}{p}})}\leq c_0.$ Hence, if  $c_{0}$ are sufficiently small such that
$Cc_0\leq \tfrac{1}{2}$, then  $$C\int_0^t\|a\|_{B_{p,1}^{\frac{d}{p}}}\| u\|_{\dot{B}_{p,1}^{1+\frac{d}{p}}(\omega)}d\tau,$$ which is the second term on the second line of \eqref{u-omega}, can be absorbed by $\|u\|_{L^1_t(\dot{B}_{p,1}^{1+\frac{d}{p}}(\omega))}$, which is the left-hand side of \eqref{u-omega}. Next,
choosing a small enough positive number $\sigma$, we consider $\sigma\times \eqref{a-omega}+\eqref{u-omega}$, then
\begin{align*}
\|a\|_{\tilde{L}_{t}^{\infty}(\dot{B}_{p,1}^{\frac{d}{p}}(\omega))}+\|u\|_{\tilde{L}_{t}^{\infty}(\dot{B}_{p,1}^{-1+\frac{d}{p}}(\omega))}&+\|u\|_{L^1_t(\dot{B}_{p,1}^{1+\frac{d}{p}}(\omega))}\leq C(\|a_{0}\|_{\dot{B}_{p,1}^{\frac{d}{p}}(\omega)}+\|u_0\|_{\dot{B}_{p,1}^{-1+\frac{d}{p}}(\omega)})\no\\&+C\int_0^t(1+\|u\|_{\dot{B}_{p,1}^{1+\frac{d}{p}}})  (\|a\|_{\dot{B}_{p,1}^{\frac{d}{p}}(\omega)}+\| u\|_{\dot{B}_{p,1}^{-1+\frac{d}{p}}(\omega)})\mathrm{d}\tau.
\end{align*}
Applying the Gronwall inequality completes the proof of Proposition \ref{au-omega-prior}.
\end{proof}

{\bf Step 2. } A difference estimate of velocity field.

The high-frequency component has been handled in Step 1. Next, we will need a difference estimate to handle the low-frequency component. A new difference estimate of velocity field in the following, which play a special role in dealing with the low-frequency component. We need to prove
\begin{equation}\label{L1Linfi}
\begin{aligned}
\|u^n-u\|_{L^1_T(L^\infty)}
\leq C \|(a_{0}^n-a_{0},u_{0}^n-u_{0})\|_{\mathbb{X}_p},
\end{aligned}
\end{equation}
here $C$ may depend on $T$.
This will be proved using the Lagrangian approach in \cite{D2014}. In fact, let $X^n$ be the flow of $u^n$ and $X$ be the flow of $u$ defined as
\begin{align*}
X^n(t,y)=y+\int_0^t u^n(\tau,X^n(\tau,y))\mathrm{d}\tau,
\end{align*}
\begin{align*}
X(t,y)=y+\int_0^t u(\tau,X(\tau,y))\mathrm{d}\tau.
\end{align*}
For all $t\in [0,T]$, the maps $X^n(t,\cdot)$ and  $X(t,\cdot)$ are both $C^1$-diffeomorphism over $\mathbb{R}^d$.
Denote
\begin{align*}
\bar{a}^n=a^n(t,X^n(t,y)),\,\, \bar{u}^n=u^n(t,X^n(t,y)),
\end{align*}
\begin{align*}
\bar{a}=a(t,X(t,y)),\,\, \bar{u}=u(t,X(t,y)).
\end{align*}
According to Proposition 3.7 in \cite{D2014}, $(\bar{a}^n,\bar{u}^n)$ belongs to $C([0,T];\dot{B}^{\frac{d}{p}}_{p,1})\times E_p(T)$, and satisfies the Lagrangian coordinate version of the system \eqref{eulercauchy}, one can see~\cite{D2014}, page 757.
Then by Theorem 2.2 in \cite{D2014}, we know the fact that the flow map of this Lagrangian version system is Lipschitz continuous from $\dot{B}_{p,1}^{\frac{d}{p}}\times\dot{B}_{p,1}^{-1+\frac{d}{p}}$ to $C([0,T];\dot{B}^{\frac{d}{p}}_{p,1})\times E_p(T)$, namely
\begin{align}\label{lag-con}
\|(\bar a^n-\bar a, \bar u^n-\bar u)\|_{Z_p(T)}\leq C\|(a_{0}^n-a_{0},u_{0}^n-u_{0})\|_{\mathbb{X}_p}.
\end{align}
Next, using \eqref {lag-con}, the interpolation inequality and the embedding $\dot{B}_{p,1}^{\frac{d}{p}}\hookrightarrow  L^\infty$, we have
\begin{equation}\label{infinite}
\begin{aligned}
\|\bar{u}^n-\bar{u}\|_{L^1_T(L^\infty)}
\lesssim T^{\frac{1}{2}} \|\bar{u}^n-\bar{u}\|_{L^2_T(\dot{B}_{p,1}^{\frac{d}{p}})} \lesssim& T^{\frac{1}{2}} \|\bar{u}^n-\bar{u}\|_{L^\infty_T(\dot{B}_{p,1}^{-1+\frac{d}{p}})}^{\frac{1}{2}}\|\bar{u}^n-\bar{u}\|_{L^1_T(\dot{B}_{p,1}^{1+\frac{d}{p}})}^{\frac{1}{2}}\\
\lesssim& T^{\frac{1}{2}}\|(a_{0}^n-a_{0},u_{0}^n-u_{0})\|_{\mathbb{X}_p}.
\end{aligned}
\end{equation}
By the definitions of $X^n,\,\, X$ and  the uniform bounds $(a^{n},u^{n})\in\mathbb{E}_{uni}$,  it can be checked that
\begin{equation}\label{unlinfin}
\begin{aligned}
\|u^n(X(t,y))-u^n(X^n(t,y))\|_{L^1_T(L^\infty)} \lesssim &\|\nabla u^n\|_{L^1_T(\dot{B}_{p,1}^{\frac{d}{p}})}\|X^n(t,y)-X(t,y)\|_{L^\infty_T(L^\infty)}\\
\lesssim& \|\bar{u}^n-\bar{u}\|_{L^1_T(L^\infty)}.
\end{aligned}
\end{equation}
Going back to the Eulerian coordinate, by \eqref{unlinfin}  and \eqref{infinite} we easily infer
\begin{equation*}
\begin{aligned}
\|u^n(y)-u(y)\|_{L^1_T(L^\infty)}=&\|u^n(X(t,y))-u(X(t,y))\|_{L^1_T(L^\infty)}\\
\leq&\|u^n(X(t,y))-u^n(X^n(t,y))\|_{L^1_T(L^\infty)}+\|\bar{u}^n-\bar{u}\|_{L^1_T(L^\infty)}\\
\lesssim& \|\bar{u}^n-\bar{u}\|_{L^1_T(L^\infty)}\les \|(a_{0}^n-a_{0},u_{0}^n-u_{0})\|_{\mathbb{X}_p},
\end{aligned}
\end{equation*}
from which, we conclude \eqref{L1Linfi}.

{\bf Step 3.} Low frequency difference estimates

\begin{prop}\label{diff-1ow}
Let $(a_1,u_1)$ and $(a_2,u_2)$
be two solutions of the system \eqref{eulercauchy} and $m_{0}\in \mathbb{Z}$. Assume that
$\|a_i\|_{L^\infty_T(\dot{B}_{p,1}^{\frac{d}{p}})}\leq \sigma_0$ for $i=1,2$
with some sufficiently small $\sigma_0$ and $T>0$. Then the low frequency part of the difference $(\delta a,\delta u)=(a_2-a_1,u_2-u_1)$ satisfies
 \begin{align*}
&\|\delta a\|_{\tilde{L}^\infty_T(\dot{B}_{p,1}^{\frac{d}{p}})}^{P_{\leq m_0}}+\|\delta u\|_{\tilde{L}^\infty_T(\dot{B}_{p,1}^{-1+\frac{d}{p}})}^{P_{\leq m_0}}
+\|\delta u\|_{L^1_T(\dot{B}_{p,1}^{1+\frac{d}{p}})}^{P_{\leq m_0}}\\
\leq& C\exp\left\{\int_0^t\alpha(\tau)\mathrm{d}\tau\right\}\Big(\| \delta a_0\|_{\dot{B}_{p,1}^{\frac{d}{p}}}^{P_{\leq m_0}}+
   \|\delta u_0\|_{\dot{B}_{p,1}^{-1+\frac{d} {p}}}^{P_{\leq m_0}}
+2^{m_0} \|\delta u\|_{L^1_T(L^\infty)}\|a_1\|_{L^\infty_t(\dot{B}_{p,1}^{\frac{d}{p}})}
\\
&\qquad +\|\delta u\|_{L^1_t(\dot{B}_{p,1}^{1+\frac{d}{p}})}^{P_{>m_0}}+\int_0^T\alpha(\tau)\mathrm{d}\tau
   \big(\|\delta a\|_{L^\infty_T(\dot{B}_{p,1}^{\frac{d}{p}})}^{P_{>m_0}}+\|\delta u\|_{L^\infty_T(\dot{B}_{p,1}^{-1+\frac{d}{p}})}^{P>_{m_0}}\big)\Big)
\end{align*}
with $\alpha(\tau)=1+ \| u_1(\tau)\|_{\dot{B}_{p,1}^{1+\frac{d}{p}}} +\| u_2(\tau)\|_{\dot{B}_{p,1}^{1+\frac{d}{p}}}+\| u_1(\tau)\|_{\dot{B}_{p,1}^{\frac{d}{p}}}^2 +\| u_2(\tau)\|_{\dot{B}_{p,1}^{\frac{d}{p}}}^2.$
\end{prop}
\begin{proof}
The difference $(\delta a,\delta u)$ reads
\begin{align}\label{energy-00}
 \left\{
 \begin{array}{l}
 \partial_t\delta a+u_2\cdot\nabla \delta a=I_1+I_2+I_3,\\
 \partial_t \delta a-\mathcal{A}\delta u=J_1+J_2+J_3+J_4+J_5,\\
 \delta a|_{t=0}=a_{20}-a_{10}, \delta u|_{t=0}=u_{20}-u_{10},
 \end{array}
 \right.
 \end{align}
 with  $I_1=-\delta u\cdot\nabla a_1$, $I_2=-(1+a_2){\rm div}\,\delta u$, $I_3
 =-\delta a{\rm div}\, u_1$, $J_1=-\delta u\cdot\nabla u_2$,
 $J_2=-u_1\cdot \nabla\delta u$,
 $J_3=-(I(a_2)-I(a_1))\mathcal{A} u_2$,
 $J_4=-I(a_1)\mathcal{A}\delta u$ and
 $J_5=-\nabla(G(a_2)-G(a_1))$.
Using the Bony decomposition, we write
\begin{align*}
    I_1=-\dot{T}_{\delta u}\nabla a_1-\dot{T}_{\nabla a_1}\delta u-\dot{R}(\delta u,\nabla a_1):=I_{11}+I_{12}+I_{13}.
\end{align*}
For $0\leq t\leq T$, applying the operator $\dot{\Delta}_j$ to the first equation of \eqref{energy-00}, taking the $L^p$ estimates, proceeding similar proof of
Lemma~\ref{transport}, we end up with
\begin{align} \label{delta j a}
\|\dot{\Delta}_j \delta a(t)\|_{L^{p}}  \leq &\|\Delta_j \delta a_{0}\|_{L^{p}}+\int_0^t \|\dot{\Delta}_j I_{11}\|_{L^p}\mathrm{d} \tau \no\\&+\int_0^t \big(\|\dot{\Delta}_j \sum_{i=2}^{3}(I_{1i}+I_{i})\|_{L^{p}}+\frac{\|{\rm div}\, u_2\|_{L^{\infty}}}{p}\|\dot{\Delta}_j \delta a\|_{L^{p}}+\|R_j\|_{L^{p}}\big) d\tau,
\end{align} where $R_j=
[u_{2}\cdot \nabla,\dot{\Delta}_j]\delta a.$
Multiplying the above inequality by $2^{j\frac{d}{p}}$, summing up over $-\infty<j\leq m_0$. Note that summing up over $-\infty<j\leq m_0$ for the second line of \eqref{delta j a} can be controlled by summing up over $-\infty<j<+\infty$.  Therefore, we conclude that
 \begin{align} \label{differ-a-low}
  \|\delta a\|_{\tilde{L}^\infty_t(\dot{B}_{p,1}^{\frac{d}{p}})}^{P_{\leq m_0}}
\leq &\| \delta a_0\|_{\dot{B}_{p,1}^{\frac{d}{p}}}^{P_{\leq m_0}}+\int_0^t \| I_{11}\|_{\dot{B}_{p,1}^{\frac{d}{p}}}^{P_{\leq m_0}}\mathrm{d} \tau
   \no\\
 &+\int_0^t(\|\sum_{i=2}^{3}(I_{1i}+I_{i})\|_{\dot{B}_{p,1}^{\frac{d}{p}}}+\|\nabla u_2\|_{\dot{B}_{p,1}^{\frac{d}{p}}}
   \|\delta a\|_{\dot{B}_{p,1}^{\frac{d}{p}}})\mathrm{d}\tau,
\end{align}
where we have used the following fact $\sum\limits_{j\in \mathbb{Z}}2^{j\frac{d}{p}}\|R_j\|_{L^p}\lesssim \|\nabla u_2\|_{\dot{B}_{p,1}^{\frac{d}{p}}}
   \|\delta a\|_{\dot{B}_{p,1}^{\frac{d}{p}}}$ (see \eqref{CR} with $s=\frac{d}{p}, r=1, w_{j}=1$ for all $j\in \mathbb{Z}$). Similarly, for $\delta u$, applying the operator $\dot{\Delta}_j$ to the second equation of \eqref{energy-00}, using the same arguments for Lemma \ref{heatm}, we have
 \begin{align}\label{differ-u-low}
  \|\delta u\|_{\tilde{L}^\infty_t(\dot{B}_{p,1}^{-1+\frac{d}{p}})}^{P_{\leq m_0}}
+\|\delta u\|_{L^1_t(\dot{B}_{p,1}^{1+\frac{d}{p}})}^{P_{\leq m_0}}\leq C
   \|\delta u_0\|_{\dot{B}_{p,1}^{-1+\frac{d}{p}}}^{P_{\leq m_0}}
+C\int_0^t\sum_{k=1}^5\|J_k\|_{\dot{B}_{p,1}^{-1+\frac{d}{p}}}\mathrm{d}\tau.
\end{align}
Using Lemma \ref{paraproduct}, Remark \ref{remark} and  Remark \ref{remark1}, we know that the product of two functions maps
$\dot{B}_{p,1}^{\frac{d}{p}}\times\dot{B}_{p,1}^{-1+\frac{d}{p}}$ in
$\dot{B}_{p,1}^{-1+\frac{d}{p}}$ for $1\leq p<2d$, and $\dot{B}_{p,1}^{\frac{d}{p}}$ is a Banach space, which are used frequently in the following. First, for $I_{11}$, we have
\begin{align*}
\sigma\|I_{11}\|_{L^1_t(
\dot{B}_{p,1}^{\frac{d}{p}})}^{P_{\leq m_0}}\leq \sigma 2^{m_0}\|I_{11}\|_{L^1_t(
\dot{B}_{p,1}^{-1+\frac{d}{p}})}
\leq C\sigma 2^{m_0} \|\delta u\|_{L^1_t(L^\infty)}\|\nabla a_1\|_{L^\infty_t(\dot{B}_{p,1}^{-1+\frac{d}{p}})}.
\end{align*}
Notice that $\|a_i\|_{L^\infty_T(\dot{B}_{p,1}^{\frac{d}{p}})}\leq \sigma_0, i=1,2$. If $\sigma_{0}$ is sufficiently small, then we have
\begin{align*}
C\|J_4\|_{L^1_t(\dot{B}_{p,1}^{-1+\frac{d}{p}})}
\leq C\|I(a_1)\|_{L^\infty_t(\dot{B}_{p,1}^{\frac{d}{p}})}
\|\mathcal{A} \,\delta u\|_{L^1_t(\dot{B}_{p,1}^{-1+\frac{d}{p}})}
\leq \frac{1}{4}\|\delta u\|_{L^1_t(\dot{B}_{p,1}^{1+\frac{d}{p}})}.
\end{align*}
If $\sigma$ is sufficiently small, by \cite[Proposition 2.3]{RH2015} and  Remark \ref{remark1} we discover that
\begin{align*}
\sigma\|I_{12}&,I_{13},I_{2}\|_{L^1_t(\dot{B}_{p,1}^{\frac{d}{p}})}\\
\leq &C\sigma  \|\nabla a_1\|_{L^\infty_t(\dot{B}_{\infty,\infty}^{-1})}\|\delta u\|_{L^1_t(\dot{B}_{p,1}^{1+\frac{d}{p}})}+\sigma(1+ \|a_2\|_{L^\infty_t(\dot{B}_{p,1}^{\frac{d}{p}})})\|{\rm div}\,\delta u\|_{L^1_t(\dot{B}_{p,1}^{\frac{d}{p}})}\\
\leq &C\sigma(1+\|a_1\|_{L^\infty_t(\dot{B}_{p,1}^{\frac{d}{p}})}+\|a_2\|_{L^\infty_t(\dot{B}_{p,1}^{\frac{d}{p}})})\|\delta u\|_{L^1_t(\dot{B}_{p,1}^{1+\frac{d}{p}})}
\leq \frac{1}{4}\|\delta u\|_{L^1_t(\dot{B}_{p,1}^{1+\frac{d}{p}})}.
\end{align*}
Next, using the interpolation inequality and the H\"{o}lder inequality, we have
\begin{align*} \label{I_1}
C\|J_1,J_2\|_{L^1_t(\dot{B}_{p,1}^{-1+\frac{d}{p}})}
\leq& C
 \int_0^t(
 \|\delta u\|_{\dot{B}_{p,1}^{\frac{d}{p}}}\|\nabla u_2\|_{\dot{B}_{p,1}^{-1+\frac{d}{p}}}+\|u_1\|_{\dot{B}_{p,1}^{\frac{d}{p}}}\|\nabla \delta u\|_{\dot{B}_{p,1}^{-1+\frac{d}{p}}})\mathrm{d}\tau\no\\
 \leq &
 C\int_0^t(
\|u_1\|_{\dot{B}_{p,1}^{\frac{d}{p}}}^2+
 \|u_2\|_{\dot{B}_{p,1}^{\frac{d}{p}}}^2)
    \|\delta u\|_{\dot{B}_{p,1}^{-1+\frac{d}{p}}}
    \mathrm{d}\tau+\frac{1}{4} \|\delta u\|_{L^1_t(\dot{B}_{p,1}^{1+\frac{d}{p}})}.
\end{align*}
Finally, the direct calculation gives
\begin{align*}
\sigma\|I_3\|_{L^1_t(\dot{B}_{p,1}^{\frac{d}{p}})}+C\|J_3,J_5\|_{L^1_t(\dot{B}_{p,1}^{-1+\frac{d}{p}})}\leq C
\|\delta a\|_{\dot{B}_{p,1}^{\frac{d}{p}}}(1+ \| u_1\|_{\dot{B}_{p,1}^{1+\frac{d}{p}}} +\| u_2\|_{\dot{B}_{p,1}^{1+\frac{d}{p}}}).
 \end{align*}
Keep in mind that
\begin{align*}
  \|\delta u\|_{L^1_t(\dot{B}_{p,1}^{1+\frac{d}{p}})}=\|\delta u\|_{L^1_t(\dot{B}_{p,1}^{1+\frac{d}{p}})}^{P_{\leq m_0}}+\|\delta u\|_{L^1_t(\dot{B}_{p,1}^{1+\frac{d}{p}})}^{P_{>m_0}}.
 \end{align*}
We observe that $\frac{3}{4}\|\delta u\|_{L^1_t(\dot{B}_{p,1}^{1+\frac{d}{p}})}^{P_{\leq m_0}}$  can be absorbed by the left hand-side of \eqref{differ-u-low}. Plugging all the above inequalities into  the inequality $\sigma\times \eqref{differ-a-low}+\eqref{differ-u-low}$ for $\sigma>0$ small enough, we conclude that
 \begin{align*}
  \|\delta a\|_{\tilde{L}^\infty_t(\dot{B}_{p,1}^{\frac{d}{p}})}^{P_{\leq m_0}}+&\|\delta u\|_{\tilde{L}^\infty_t(\dot{B}_{p,1}^{-1+\frac{d}{p}})}^{P_{\leq m_0}}
+\|\delta u\|_{L^1_t(\dot{B}_{p,1}^{1+\frac{d}{p}})}^{P_{\leq m_0}}\no\\
\leq &C\big(\| \delta a_0\|_{\dot{B}_{p,1}^{\frac{d}{p}}}^{P_{\leq m_0}}+
   \|\delta u_0\|_{\dot{B}_{p,1}^{-1+\frac{d}{p}}}^{P_{\leq m_0}}
+2^{m_0} \|\delta u\|_{L^1_t(L^\infty)}\|a_1\|_{L^\infty_t(\dot{B}_{p,1}^{\frac{d}{p}})}\no\\
&+\|\delta u\|_{L^1_t(\dot{B}_{p,1}^{1+\frac{d}{p}})}^{P_{>m_0}}
 +\int_0^t\alpha(\tau)
   (\|\delta a\|_{\dot{B}_{p,1}^{\frac{d}{p}}}+\|\delta u\|_{\dot{B}_{p,1}^{-1+\frac{d}{p}}})\mathrm{d}\tau\big)
\end{align*}
with $\alpha(\tau)=1+ \| u_1(\tau)\|_{\dot{B}_{p,1}^{1+\frac{d}{p}}} +\| u_2(\tau)\|_{\dot{B}_{p,1}^{1+\frac{d}{p}}}+\| u_1(\tau)\|_{\dot{B}_{p,1}^{\frac{d}{p}}}^2 +\| u_2(\tau)\|_{\dot{B}_{p,1}^{\frac{d}{p}}}^2.$
Again we  divide  $ \|\delta a\|_{\dot{B}_{p,1}^{\frac{d}{p}}}$ and $\|\delta u\|_{\dot{B}_{p,1}^{-1+\frac{d}{p}}}$ for above term into the following parts: \begin{align*}
  \|\delta a\|_{\dot{B}_{p,1}^{\frac{d}{p}}}=\|\delta a\|_{\dot{B}_{p,1}^{\frac{d}{p}}}^{P_{\leq m_0}}+\|\delta a\|_{\dot{B}_{p,1}^{\frac{d}{p}}}^{P_{>m_0}},\,
   \|\delta u\|_{\dot{B}_{p,1}^{-1+\frac{d}{p}}}=\|\delta u\|_{\dot{B}_{p,1}^{-1+\frac{d}{p}}}^{P_{\leq m_0}}+\|\delta u\|_{\dot{B}_{p,1}^{-1+\frac{d}{p}}}^{P_{>m_0}}.
 \end{align*}
Then applying the Gronwall inequality for $0\leq t\leq T$ completes the proof of Proposition \ref{diff-1ow}.
\end{proof}

We are ready to prove that $(a^n,u^n)$ converges to $(a,u)$ in $Z_p(T)$. For any $\varepsilon>0$, by Step 1, there exists a $K>0$ such that
\begin{equation}\label{highfrediffer}
\begin{aligned}
  \|(a^n-a, u^n-u)\|_{Z_p(T)}^{P_{>K}}\lesssim\varepsilon, \quad \forall\, n.
\end{aligned}
\end{equation}
Fixing $K$, by Proposition \ref{diff-1ow} in Step 3, we have
\begin{equation*}
\begin{aligned}
  \|(a^n-a, u^n-u)\|_{Z_p(T)}^{P_{\leq K}}\lesssim
\|(a_0^n-a_0, u^n_0-u_0)\|_{\mathbb{X}_{p}}+2^{K}\|u^n-u\|_{L^1_T(L^\infty)}
+\varepsilon,
\end{aligned}
\end{equation*}
from which, by \eqref{L1Linfi} in Step 2, we get
\begin{equation}\label{lowfrediffer}
\begin{aligned}
  \|(a^n-a, u^n-u)\|_{Z_p(T)}^{P_{\leq K}}\lesssim (2^{K}+1)
\|(a_0^n-a_0, u^n_0-u_0)\|_{\mathbb{X}_{p}}
+\varepsilon.
\end{aligned}
\end{equation}
By virtue of the convergence of the initial data,  there exits a positive number $N_1$, if $n>N_1$, then \begin{equation}\label{initialconven}
\begin{aligned}
  \|a_0^n-a_0\|_{\dot{B}_{p,1}^{\frac{d}{p}}}
+\|u_0^n-u_0\|_{\dot{B}_{p,1}^{-1+\frac{d}{p}}}\lesssim \frac{\varepsilon}{2^{K}+1}.
\end{aligned}
\end{equation}
Combining
\eqref{lowfrediffer} and  \eqref{initialconven} yields that there exists $M=M(K,N_1, \varepsilon)>0$ such that if $n>M$, then
\begin{equation}\label{lowerdiffer}
\begin{aligned}
  \|(a^n-a, u^n-u)\|_{Z_p(T)}^{P_{\leq K}}\lesssim\varepsilon.
\end{aligned}
\end{equation}
Therefore taking advantage of  \eqref{highfrediffer} and  \eqref{lowerdiffer}, we end up with
\begin{equation*}
\begin{aligned}
\|(a^n-a, u^n-u)\|_{Z_p(T)}\lesssim \varepsilon, \quad \forall\, n>M.
\end{aligned}
\end{equation*}
This implies  $(a^n,u^n)$ converges to $(a,u)$ in $Z_p(T)$.

\appendix

\section{Bona-Smith method: $1\leq p<d$}\label{proofapelwp}

In the appendix, we prove the continuity of solution map to the compressible Navier-Stokes equations \eqref{eulercauchy} in the critical Besov space $\dot{B}_{p,1}^{\frac{d}{p}}\times \dot{B}_{p,1}^{-1+\frac{d}{p}}$ for $1 \leq p < d$, using Bona-Smith argument \cite{Bona-Smith}. This is to illustrate the main difficulties we encountered and the advantage of using the frequency envelope method.
\begin{theorem} \label{apelwp}
Let $1 \leq p <d,\, d\geq3,\,  (a_{0}, u_{0})\in \mathbb{X}_p$ with $
\|a_{0}\|_{\dot{B}_{p,1}^{\frac{d}{p}}}\leq c$ for a small enough constant $c>0$.
Then there exists a neighborhood $U$ of $(a_{0}, u_{0})$ in $\mathbb{X}_p$ and $T=T(U)>0$, such that for any data $(\tilde{a}_0, \tilde{u}_0)\in U$, the Cauchy problem \eqref{eulercauchy} has a unique solution
  $$(\tilde{a}, \tilde{u}):=S_{T}(\tilde{a}_0,\tilde{u}_0)\in C([0,T];\dot{B}_{p,1}^{\frac{d}{p}})\times(C([0,T];\dot{B}_{p,1}^{-1+\frac{d}{p}})\cap L^{1}([0,T]; \dot{B}_{p,1}^{1+\frac{d}{p}})).$$
Moreover, the solution map $S_T$ is continuous from $U$ to $C([0,T];\X_p)$. Precisely, for any $\varepsilon>0$, there exists $\tilde{\delta}_{0}>0$ such that for any $(\tilde{a}_0, \tilde{u}_0) \in U$ with  $\|a_0 -\tilde{a}_0\|_{{B}_{p,1}^{\frac{d}{p}}} + \|u_{0}-\tilde{u}_0\|_{{B}_{p,1}^{-1+\frac{d}{p}}} \leq\tilde{\delta}_{0}$, then
  $
\|S_{T}(a_{0}, u_0)-S_{T}(\tilde{a}_{0}, \tilde{u}_{0})\|_{Z_p(T)}\leq\varepsilon
$.
\end{theorem}
\begin{rem}\label{globalcon}
We remark that using Bona-Smith argument~\cite{Bona-Smith}, we may obtain the solution map $S_T$ is continuous in some hybrid-Besov spaces for any $T>0$. The reader may refer to \cite{FD2010, CMZ2010, H2011, RH2015} for the global existence and uniqueness. Precisely, let  $k_{0}\in\mathbb{Z},\, s\in\mathbb{R},\,  z^{\ell}=\dot{S}_{k_{0}+1}z$ and $z^{h}=z-z^{\ell}$.
We endow those spaces with the norms:
\begin{equation*}
\begin{aligned}
&\|(a,u)\|_{Y_{p}}:=\|(a,u)^{\ell}\|_{\tilde{L}_{T}^{\infty}(\dot{B}_{2,1}^{-1
+\frac{d}{2}})}+
\|(a,u)^{\ell}\|_{\tilde{L}_{T}^{1}(\dot{B}_{2,1}^{1+\frac{d}{2}})}
\\
& \ \ +
\|a^{h}\|_{\tilde{L}_{T}^{\infty}(\dot{B}_{p,1}^{\frac{d}{p}})}+
\|a^{h}\|_{\tilde{L}_{T}^{1}(\dot{B}_{p,1}^{\frac{d}{p}})}+
\|u^{h}\|_{\tilde{L}_{T}^{\infty}(\dot{B}_{p,1}^{-1+\frac{d}{p}})}+
\|u^{h}\|_{\tilde{L}_{T}^{1}(\dot{B}_{p,1}^{1+\frac{d}{p}})}\\
\end{aligned}
\end{equation*}
and
\begin{equation*}
\begin{aligned}
\|(a_{0},u_{0})\|_{Y_{p}^{0}}:=\|(a_{0},u_{0})^{\ell}\|_{\dot{B}_{2,1}^{-1+\frac{d}{2}}}
+\|a_{0}^{h}\|_{\dot{B}_{p,1}^{\frac{d}{p}}}+\|u_{0}^{h}\|_{\dot{B}_{p,1}^{-1+\frac{d}{p}}}.
\end{aligned}
\end{equation*}
Let $d\geq3,\, 2\leq p\leq\min\{4,\frac{2d}{d-2}\}, p<d, (\tilde{u}_{0}, \tilde{a}_{0})\in\mathbb{X}_p$ with besides $(\tilde{a}_{0}^{\ell},\tilde{u}_{0}^{\ell})\in\dot{B}_{2,1}^{-1+\frac{d}{2}}$ satisfy
$\|(\tilde{a}_{0}, \tilde{u}_{0})\|_{Y_{p}^{0}}\leq c$ for a small constant $c>0$
and  some $k_{0}\in \mathbb{Z}$. There exists a neighbourhood $U$ of $(\tilde{a}_{0}, \tilde{u}_{0})$ in $Y_{p}^{0}$, such that for any data $(a_0, u_0)\in U$, the Cauchy problem \eqref{eulercauchy} has a unique solution
$(a, u):=S_{T}(a_0,u_0)\in Y_{p}$
for any $T>0$. Moreover, the solution map $S_T$ is continuous from $U$ to $Y_{p}$ for any $T>0$.
\end{rem}

We first give some difference estimates to the compressible Navier-Stokes equations \eqref{eulercauchy}.
\begin{prop}\label{diff-1-local}
Let $ d\geq 3$, $1\leq p<d$, and let
$(a_i,u_i)\in C([0,T];\dot{B}_{p,1}^{\frac{d}{p}})\times \big(C([0,T];\dot{B}_{p,1}^{-1+\frac{d}{p}})\cap L^1([0,T];\dot{B}_{p,1}^{1+\frac{d}{p}})\big)$
be two solutions of the system \eqref{eulercauchy} with
the initial data $(a_{0i},u_{0i})\in \dot{B}_{p,1}^{\frac{d}{p}}\times \dot{B}_{p,1}^{-1+\frac{d}{p}}$, $i=1, 2$.
Assume that
$\|a_1\|_{L^\infty_T(\dot{B}_{p,1}^{\frac{d}{p}})}\leq c_{0}$
for some sufficient small $c_{0}$.

$(1)$ If $a_1\in L^2_T(\dot{B}_{p,1}^{1+\frac{d}{p}})$, then for the difference $(\delta a,\delta u)=(a_2-a_1,u_2-u_1)$, we have
  \begin{equation}\label{lodiffb}
\begin{aligned}
    \|\delta a\|_{L^\infty_T(\dot{B}_{p,1}^{\frac{d}{p}})}
   +\|\delta u&\|_{L^\infty_T(\dot{B}_{p,1}^{-1+\frac{d}{p}})}+\|\delta u\|_{ L^1_T(\dot{B}_{p,1}^{1+\frac{d}{p}})}\\
   \leq &C\exp\big\{C\int_0^T\|a_1(\tau)\|_{\dot{B}_{p,1}^
   {1+\frac{d}{p}}}^2\mathrm{d}\tau \big\} (\|\delta a_0\|_{\dot{B}_{p,1}^{\frac{d}{p}}}
   +\|\delta u_0\|_{\dot{B}_{p,1}^{-1+\frac{d}{p}}}).
\end{aligned}
\end{equation}

$(2)$ If $a_1\in L^\infty_T(\dot{B}_{p,1}^{1+\frac{d}{p}})$, and in addition,  $(\delta a_0,\delta u_0)\in \dot{B}_{p,1}^{-1+\frac{d}{p}}\times \dot{B}_{p,1}^{-2+\frac{d}{p}}$, then we obtain
   \begin{equation}\label{lodiffc}
\begin{aligned}
    \|\delta& a\|_{L^\infty_T(\dot{B}_{p,1}^{\frac{d}{p}})}
   +\|\delta u\|_{L^\infty_T(\dot{B}_{p,1}^{-1+\frac{d}{p}})}+
   \|\delta u\|_{L^1_T(\dot{B}_{p,1}^{1+\frac{d}{p}})}
  \\ \leq &C(\|\delta a_0\|_{\dot{B}_{p,1}^{\frac{d}{p}}}
   +\|\delta u_0\|_{\dot{B}_{p,1}^{-1+\frac{d}{p}}}+\|a_1\|_{L^\infty_T(\dot{B}_{p,1}^{1+\frac{d}{p}})}
   (\|\delta a_0\|_{\dot{B}_{p,1}^{-1+\frac{d}{p}}}
   +\|\delta u_0\|_{\dot{B}_{p,1}^{-2+\frac{d}{p}}})\big).
 \end{aligned}
\end{equation}
Here $C$ depends on $T$, $\|u_i(t)\|_{L^\infty_T(\dot{B}_{p,1}^{-1+\frac{d}{p}})
\cap L^1_T(\dot{B}_{p,1}^{1+\frac{d}{p}})}$ and $\|a_i\|_{L^\infty_T(\dot{B}_{p,1}^{\frac{d}{p}})}$ 
 for $i=1, 2$.
  \end{prop}

\begin{proof}
From the compressible Navier-Stokes equations \eqref{eulercauchy}, the difference $(\delta a,\delta u)$ reads
\begin{align}\label{difftrans}
 \left\{
 \begin{array}{l}
 \partial_t\delta a+u_2\cdot\nabla \delta a=I_1+I_2+I_3,\\
 \partial_t \delta a-\mathcal{A}\delta u=J_1+J_2+J_3+J_4+J_5,\\
 \delta a|_{t=0}=a_{02}-a_{01}, \delta u|_{t=0}=u_{02}-u_{01},
 \end{array}
 \right.
 \end{align}
 with  $I_1=-\delta u\cdot\nabla a_1$, $I_2=-(1+a_2){\rm div}\,\delta u$, $I_3
 =-\delta a{\rm div}\, u_1$, $J_1=-\delta u\cdot\nabla u_2$,
 $J_2=-u_1\cdot \nabla\delta u$,
 $J_3=-(I(a_2)-I(a_1))\mathcal{A} u_2$,
 $J_4=-I(a_1)\mathcal{A}\delta u$ and
 $J_5=-\nabla(G(a_2)-G(a_1))$.

$\bullet$ {\bf  Case (1);} When $a_1\in L^2_T(\dot{B}_{p,1}^{1+\frac{d}{p}})$.
Applying Lemma \ref{transport} and  Remark \ref{remark1} for $\delta a$ and  the maximal regularity of heat equation (see Corollary \ref{cor:heatm} and Remark \ref{remark1}) for $\delta u$, and choosing a positive small number $\sigma$ (which will be specified later) to combine the estimates for $\delta a $ and $\delta u$,  we have, for $0\leq t\leq T$,
\begin{align}\label{energy-0}
  \sigma\|\delta a&\|_{L^\infty_t(\dot{B}_{p,1}^{\frac{d}{p}})}
  +\|\delta u\|_{L^\infty_t(\dot{B}_{p,1}^{-1+\frac{d}{p}})}
+\|\delta u\|_{L^1_t(\dot{B}_{p,1}^{1+\frac{d}{p}})}\lesssim \sigma \|\delta a_0\|_{\dot{B}_{p,1}^{\frac{d}{p}}}
   +\|\delta a_0\|_{\dot{B}_{p,1}^{-1+\frac{d}{p}}}\no\\
   &+\int_0^t (\sigma\|u_2\|_{\dot{B}_{p,1}^{1+\frac{d}{p}}}
   \|\delta a\|_{\dot{B}_{p,1}^{\frac{d}{p}}}
 +\sigma \sum_{k=1}^3\|I_k\|_{\dot{B}_{p,1}^{\frac{d}{p}}}
 +\sum_{k=1}^5\|J_k\|_{\dot{B}_{p,1}^{-1+\frac{d}{p}}})\mathrm{d}\tau.
\end{align}
We control $I_i$ for $i=1,\,2,\,3$ by using the fact the product of two functions in $\dot{B}_{p,1}^{\frac{d}{p}}$ is also in $\dot{B}_{p,1}^{\frac{d}{p}}$, and control $J_i$ for $i=1,\,2,\,3,\,4$ by using the product estimates
$\|fg\|_{\dot{B}_{p,1}^{-1+\frac{d}{p}}}\lesssim \|f\|_{\dot{B}_{p,1}^{-1+\frac{d}{p}}}\|g\|_{\dot{B}_{p,1}^{\frac{d}{p}}}$ ( for more details, see Lemma \ref{paraproduct}, Remark \ref{remark} and Remark \ref{remark1} for $1\leq p<2d$). For $J_5$, we take advantage of Lemma ~\ref{function1}. Therefore, it can be checked that
\begin{align*} \label{I_1}\sigma\|I_1\|_{L^1_t(\dot{B}_{p,1}^{\frac{d}{p}})}&+\|J_1\|_{L^1_t(\dot{B}_{p,1}^{-1+\frac{d}{p}})}
+\|J_2\|_{L^1_t(\dot{B}_{p,1}^{-1+\frac{d}{p}})}\no\\
\lesssim &
 \int_0^t(\|\delta u\|_{\dot{B}_{p,1}^{\frac{d}{p}}}\|\nabla a_1\|_{\dot{B}_{p,1}^{\frac{d}{p}}}
 +\|\delta u\|_{\dot{B}_{p,1}^{\frac{d}{p}}}\|\nabla u_2\|_{\dot{B}_{p,1}^{-1+\frac{d}{p}}}+\|u_1\|_{\dot{B}_{p,1}^{\frac{d}{p}}}\|\nabla \delta u\|_{\dot{B}_{p,1}^{-1+\frac{d}{p}}})\mathrm{d}\tau\no\\
 \leq &
 C\int_0^t(
 \|a_1\|_{\dot{B}_{p,1}^{1+\frac{d}{p}}}^2+\|u_1\|_{\dot{B}_{p,1}^{\frac{d}{p}}}^2+
 \|u_2\|_{\dot{B}_{p,1}^{\frac{d}{p}}}^2)
    \|\delta u\|_{\dot{B}_{p,1}^{-1+\frac{d}{p}}}
    \mathrm{d}\tau+\frac{1}{4} \|u_2-u_1\|_{L^1_t(\dot{B}_{p,1}^{1+\frac{d}{p}})},
\end{align*}
which we have used the interpolation inequality and the H\"{o}lder inequality. Next, recall that $\|a_1\|_{L^\infty_T(\dot{B}_{p,1}^{\frac{d}{p}})}\leq c_{0}$. Hence, if $\sigma$ and $c_{0}$ are sufficiently small, then we have that
\begin{align*}
\sigma\|I_2\|_{L^1_t(\dot{B}_{p,1}^{\frac{d}{p}})}+&\|J_4\|_{L^1_t(\dot{B}_{p,1}^{-1+\frac{d}{p}})}\\
\lesssim &\sigma(1+ \|a_2\|_{L^\infty_t(\dot{B}_{p,1}^{\frac{d}{p}})})\|{\rm div}\,\delta u\|_{L^1_t(\dot{B}_{p,1}^{\frac{d}{p}})}+\|I(a_1)\|_{L^\infty_t(\dot{B}_{p,1}^{\frac{d}{p}})}
\|\mathcal{A} \,\delta u\|_{L^1_t(\dot{B}_{p,1}^{-1+\frac{d}{p}})}\\
\leq &C(\sigma+c_{0})\|\delta u\|_{L^1_t(\dot{B}_{p,1}^{1+\frac{d}{p}})}
\leq \frac{1}{4} \|\delta u\|_{L^1_t(\dot{B}_{p,1}^{1+\frac{d}{p}})}.
\end{align*}
Finally, the direct calculation gives
\begin{equation*}
\begin{aligned}
    &\sigma\|I_3\|_{L^1_t(\dot{B}_{p,1}^{\frac{d}{p}})}+
    \|J_3\|_{L^1_t(\dot{B}_{p,1}^{-1+\frac{d}{p}})}+
 \|J_5\|_{L^1_t(\dot{B}_{p,1}^{-1+\frac{d}{p}})}
 \\
& \  \ \lesssim
\int_{0}^{t}\|\delta a\|_{\dot{B}_{p,1}^{\frac{d}{p}}}(1+\| u_1\|_{\dot{B}_{p,1}^{1+\frac{d}{p}}} +\| u_2\|_{\dot{B}_{p,1}^{1+\frac{d}{p}}})d\tau.
\end{aligned}
\end{equation*}
 Then plugging the above estimates into the inequality \eqref{energy-0},  and applying Gronwall's inequality, we get the desired result \eqref{lodiffb}.

$\bullet$ {\bf Case (2)}: Let $a_1\in L^\infty_T(\dot{B}_{p,1}^{1+\frac{d}{p}})$ and $(\delta a_0,\delta u_0)\in \dot{B}_{p,1}^{-1+\frac{d}{p}}\times \dot{B}_{p,1}^{-2+\frac{d}{p}}$.  First, according to the assumption $(a_i,u_i)\in C([0,T];\dot{B}_{p,1}^{\frac{d}{p}})
\times \big(C([0,T];\dot{B}_{p,1}^{-1+\frac{d}{p}})\cap
L^1([0,T];\dot{B}_{p,1}^{1+\frac{d}{p}})\big)$, it can be checked that
the right-hand side terms of $\eqref{difftrans}_{1}$ and  $\eqref{difftrans}_{2}$ belong to  $L^1([0,T];{\dot{B}_{p,1}^{-1+\frac{d}{p}}})$ and $L^1([0,T];{\dot{B}_{p,1}^{-2+\frac{d}{p}}})$, respectively. Recall that  $(\delta a_0,\delta u_0)\in  \dot{B}_{p,1}^{-1+\frac{d}{p}}\times \dot{B}_{p,1}^{-2+\frac{d}{p}}$. Hence, the transport equation and the heat equation theory (see Lemma \ref{transport}, Corollary \ref{cor:heatm} and
Remark \ref{remark1}) yields
  $(\delta a,\delta u) \in C([0,T];\dot{B}_{p,1}^{-1+\frac{d}{p}})\times( C([0,T];\dot{B}_{p,1}^{-2+\frac{d}{p}})\cap L^1([0,T];\dot{B}_{p,1}^{\frac{d}{p}})).$  In
  the same manner, similar to \eqref{energy-0}, we have
  \begin{equation}\label{energy-1}
\begin{aligned}
  \sigma\|\delta &a\|_{L^\infty_t(\dot{B}_{p,1}^{-1+\frac{d}{p}})}
  +\|\delta u\|_{L^\infty_t(\dot{B}_{p,1}^{-2+\frac{d}{p}})}
+\|\delta u\|_{L^1_t(\dot{B}_{p,1}^{\frac{d}{p}})}\lesssim \sigma \|\delta a_0\|_{\dot{B}_{p,1}^{-1+\frac{d}{p}}}
   +\|\delta a_0\|_{\dot{B}_{p,1}^{-2+\frac{d}{p}}}\\
  & +\int_0^t (\sigma\|u_2\|_{\dot{B}_{p,1}^{1+\frac{d}{p}}}
   \|\delta a\|_{\dot{B}_{p,1}^{-1+\frac{d}{p}}}
 +\sigma \sum_{k=1}^3\|I_k\|_{\dot{B}_{p,1}^{-1+\frac{d}{p}}}
 +\sum_{k=1}^5\|J_k\|_{\dot{B}_{p,1}^{-2+\frac{d}{p}}})\mathrm{d}\tau.
\end{aligned}
\end{equation}
Thanks to Lemma \ref{paraproduct}, Remark \ref{remark} and Remark \ref{remark1}  (using the conditions $d\geq 3$ and $1\leq p<d$), for $I_i$ $(i=1,\, 2,\, 3)$ we use the product estimate $\|fg\|_{\dot{B}_{p,1}^{-1+\frac{d}{p}}}\lesssim \|f\|_{\dot{B}_{p,1}^{-1+\frac{d}{p}}}\|g\|_{\dot{B}_{p,1}^{\frac{d}{p}}}$; Likewise for $J_i (i=1,\, 2,\, 3,\, 4)$ we use
$\|fg\|_{\dot{B}_{p,1}^{-2+\frac{d}{p}}}\lesssim \|f\|_{\dot{B}_{p,1}^{-2+\frac{d}{p}}}\|g\|_{\dot{B}_{p,1}^{\frac{d}{p}}}$
or
$\|fg\|_{\dot{B}_{p,1}^{-2+\frac{d}{p}}}\lesssim \|f\|_{\dot{B}_{p,1}^{-1+\frac{d}{p}}}\|g\|_{\dot{B}_{p,1}^{-1+\frac{d}{p}}}$. From Remark \ref{remark1} and Lemma ~\ref{function1} with $s=-1+\frac{d}{p}>0$ (using the condition $1\leq p<d$), we estimate $J_5$.
These arguments yield, if $\sigma$ and $c_{0}$ are sufficiently small,
\begin{equation*}
\begin{aligned}
\sigma\|I_1\|_{L^1_t(\dot{B}_{p,1}^{-1+\frac{d}{p}})}
+&\sigma\|I_2\|_{L^1_t(\dot{B}_{p,1}^{-1+\frac{d}{p}})}
+\|J_4\|_{L^1_t(\dot{B}_{p,1}^{-2+\frac{d}{p}})}\\
\lesssim &(\sigma(1+ \|a_1\|_{L^\infty_T(\dot{B}_{p,1}^{\frac{d}{p}})}+\|a_2\|_{L^\infty_T(\dot{B}_{p,1}^{\frac{d}{p}})})
+\|a_1\|_{L^\infty_t(\dot{B}_{p,1}^{\frac{d}{p}})})
\|\delta u\|_{L^1_t(\dot{B}_{p,1}^{\frac{d}{p}})}\\
\leq &C(\sigma+\delta_{0})\|\delta u\|_{L^1_t(\dot{B}_{p,1}^{\frac{d}{p}})}
\leq \frac{1}{4} \|\delta u\|_{L^1_t(\dot{B}_{p,1}^{\frac{d}{p}})}
\end{aligned}
\end{equation*}
and
\begin{equation*}
\begin{aligned}
\|J_1\|_{L^1_t(\dot{B}_{p,1}^{-2+\frac{d}{p}})}+
\|J_2\|_{L^1_t(\dot{B}_{p,1}^{-2+\frac{d}{p}})}
&  \lesssim \int_0^t(\| u_2\|_{\dot{B}_{p,1}^{\frac{d}{p}}}+\|u_1\|_{\dot{B}_{p,1}^{\frac{d}{p}}})\| \delta u\|_{\dot{B}_{p,1}^{-1+\frac{d}{p}}}\mathrm{d}\tau\\
 &  \leq C\int_0^t
(\|u_1\|_{\dot{B}_{p,1}^{\frac{d}{p}}}^2+\|u_2\|_{\dot{B}_{p,1}^{\frac{d}{p}}}^2)
     \|\delta u\|_{\dot{B}_{p,1}^{\frac{d}{p}-2}}\mathrm{d}\tau
    +\frac{1}{4} \|\delta u\|_{L^1_t(\dot{B}_{p,1}^{\frac{d}{p}})}
\end{aligned}
\end{equation*}
as well as
\begin{equation*}
\begin{aligned}
\sigma\|I_3\|_{L^1_t(\dot{B}_{p,1}^{-1+\frac{d}{p}})}+&
 \|J_3\|_{L^1_t(\dot{B}_{p,1}^{-2+\frac{d}{p}})}+
 \|J_5\|_{L^1_t(\dot{B}_{p,1}^{-2+\frac{d}{p}})}\\&\lesssim
\int_{0}^{t}\|\delta a\|_{\dot{B}_{p,1}^{-1+\frac{d}{p}}}(1+\| u_1\|_{\dot{B}_{p,1}^{1+\frac{d}{p}}} +\| u_2\|_{\dot{B}_{p,1}^{1+\frac{d}{p}}})d\tau.
\end{aligned}
\end{equation*}
Substituting above estimates into \eqref{energy-1}, then applying Gronwall's inequality, we conclude that
\begin{align}\label{energy-1-result}
    \|\delta a&\|_{L^\infty_T(\dot{B}_{p,1}^{-1+\frac{d}{p}})}
   +\|\delta u\|_{L^\infty_T(\dot{B}_{p,1}^{-2+\frac{d}{p}})}+
    \|\delta u\|_{L^1_T(\dot{B}_{p,1}^{\frac{d}{p}})}
   \leq C(\|\delta a_0\|_{\dot{B}_{p,1}^{-1+\frac{d}{p}}}
   +\|\delta u_0\|_{\dot{B}_{p,1}^{-2+\frac{d}{p}}}).
\end{align}
Now go back to \eqref{energy-0} again. We only retreat $I_1$, since the other terms can be  controlled in the same way as that have been down in Case 1. Using \eqref{energy-1-result}, we see
\begin{align*}
\|I_1\|_{L^1_T(\dot{B}_{p,1}^{\frac{d}{p}})}\lesssim &\|a_1\|_{L^\infty_T(\dot{B}_{p,1}^{1+\frac{d}{p}})}\|\delta u\|_{L^1_T(\dot{B}_{p,1}^{\frac{d}{p}})}\lesssim \|a_1\|_{L^\infty_T(\dot{B}_{p,1}^{1+\frac{d}{p}})}(\|\delta a_0\|_{\dot{B}_{p,1}^{-1+\frac{d}{p}}}
   +\|\delta u_0\|_{\dot{B}_{p,1}^{-2+\frac{d}{p}}}),
\end{align*}
from which and  \eqref{energy-0}, we obtain
\begin{equation*}\label{energy-3}
\begin{aligned}
\|\delta a&\|_{L^\infty_T(\dot{B}_{p,1}^{\frac{d}{p}})}
  +\|\delta u\|_{L^\infty_T(\dot{B}_{p,1}^{-1+\frac{d}{p}})}
+\|\delta u\|_{L^1_T(\dot{B}_{p,1}^{1+\frac{d}{p}})}\lesssim  \|\delta a_0\|_{\dot{B}_{p,1}^{\frac{d}{p}}}
   +\|\delta a_0\|_{\dot{B}_{p,1}^{-1+\frac{d}{p}}}\no\\
   &+ \|a_1\|_{L^\infty_T(\dot{B}_{p,1}^{1+\frac{d}{p}})}(\|\delta a_0\|_{\dot{B}_{p,1}^{-1+\frac{d}{p}}}
   +\|\delta u_0\|_{\dot{B}_{p,1}^{-2+\frac{d}{p}}})\no\\
   &+\int_{0}^{T}\|(\delta a,\delta u)\|_{\dot{B}_{p,1}^{\frac{d}{p}}\times\dot{B}_{p,1}^{-1+\frac{d}{p}}}(1+\| u_1\|_{\dot{B}_{p,1}^{1+\frac{d}{p}}} +\| u_2\|_{\dot{B}_{p,1}^{1+\frac{d}{p}}}+\|u_1\|_{\dot{B}_{p,1}^{\frac{d}{p}}}^2+
\|u_2\|_{\dot{B}_{p,1}^{\frac{d}{p}}}^2)d\tau.
 \end{aligned}
\end{equation*}
Applying the Gronwall inequality, we eventually obtain the desired result \eqref{lodiffc}. This completes the proof of Proposition~\ref{diff-1-local}.
\end{proof}

\noindent {\bf Proof of Theorem~\ref{apelwp}.}
Let $(a_0,u_0) \in \dot{B}_{p,1}^{\frac{d}{p}}\times\dot{B}_{p,1}^{-1+\frac{d}{p}}$ and  $\|a_{0}\|_{\dot{B}_{p,1}^{\frac{d}{p}}}\leq c$ with  $c$ small enough.
According to the local existence theory (see ~\cite{D2014,RH2015}), the Cauchy problem \eqref{eulercauchy} has unique solution $(a,u):=
S_{T}(a_0,u_0)\in\mathbb{E}_{uni}$. Following from the same arguments with Step 1 of proving Theorem~\ref{lwp}.
Then there exists a neighbourhood $U$ of $(a_{0}, u_{0})$ in $\mathbb{X}_p$ and
a constant $\eta>0$ and a uniform time $T: =T(U)>0$ such that for any $(\tilde{a}_0,\tilde{u}_0)\in U$ with  
\begin{equation}\label{Diffau}
\begin{aligned}
\|(\tilde{a_0},\tilde{u_0})-(a_0,u_0)\|_{\dot{B}_{p,1}^{\frac{d}{p}}\times \dot{B}_{p,1}^{-1+\frac{d}{p}}}\leq \eta,
\end{aligned}
\end{equation} 
the solution $(\tilde{a},\tilde{u}):=
S_{T}(\tilde{a}_0,\tilde{u}_0)\in\mathbb{E}_{uni}$ also exists and is unique. In the following, we shall show the continuous dependence on the data to the compressible Navier-Stokes equations for $1 \leq p < d$ using the Bona-Smith approximation.
The proof is divided  into three steps: \\

{\bf  Step 1. Persistence estimate. }
Let  $$(a_{N}, u_{N}):=S_{T}(P_{\leq N}a_0, P_{\leq N}u_0),  (\tilde{a}_{N}, \tilde{u}_{N}):=S_{T}(P_{\leq N}\tilde{a}_0, P_{\leq N}\tilde{u}_0)$$ be solutions of compressible Navier-Stokes equation \eqref{eulercauchy} with the initial data $(P_{\leq N}a_0, P_{\leq N}u_0)$ and $(P_{\leq N}\tilde{a}_0, P_{\leq N}\tilde{u}_0)$, respectively. Then the arguments above imply that $(a_{N}, u_{N})\in\mathbb{E}_{uni}$ and  $(\tilde{a}_{N}, \tilde{u}_{N})\in\mathbb{E}_{uni}$.

Thanks to Lemma \ref{paraproduct}, Remark \ref{remark}, Lemma \ref{transport} and Remark \ref{remark1}, we get
\begin{equation}\label{AA2.6}
\begin{aligned}
\|a_{N}\|_{L_{T}^{\infty}(\dot{B}_{p,1}^{1+\frac{d}{p}})}&\lesssim
\|P_{\leq N}a_0\|_{\dot{B}_{p,1}^{1+\frac{d}{p}}}+\|(1+a_{N}){\rm div}\, u_{N}\|_{L_{T}^{1}(\dot{B}_{p,1}^{1+\frac{d}{p}})}\\
&\   \   \  +\int_{0}^{T}\|\nabla u_{N}\|_{ \dot{B}_{p,1}
^{\frac{d}{p}}}\|a_{N}\|_{\dot{B}_{p,1}^{1+\frac{d}{p}}}dt\\
&\lesssim
\|P_{\leq N}a_0\|_{\dot{B}_{p,1}^{1+\frac{d}{p}}}+(1+\|a_{N}\|_{L_{T}^{\infty}\dot{B}_{p,1}
^{\frac{d}{p}}})
\int_{0}^{T}\|u_{N}\|_{\dot{B}_{p,1}^{2+\frac{d}{p}}}dt
\\
&\   \   \  +\int_{0}^{T}\|\nabla u_{N}\|_{ \dot{B}_{p,1}
^{\frac{d}{p}}}\|a_{N}(\tau)\|_{\dot{B}_{p,1}^{1+\frac{d}{p}}}dt.
\end{aligned}
\end{equation}
Using the maximal regularity estimates for the heat equation (see Corollary \ref{cor:heatm}), Lemma \ref{function1}, Remark \ref{remark} and Remark \ref{remark1}, we end up with
\begin{equation}\label{A2.11}
\begin{aligned}
\|u_{N}\|_{L_{T}^{\infty}(\dot{B}_{p,1}^{\frac{d}{p}})}&+\int_{0}^{T}
\|u_{N}\|_{\dot{B}_{p,1}^{2+\frac{d}{p}}}dt
  \leq C\|P_{\leq N}u_0\|_{\dot{B}_{p,1}^{\frac{d}{p}}}+C\int_{0}^{T}\|
u_{N}\|_{\dot{B}_{p,1}^{\frac{d}{p}}}\|\nabla u_{N}\|_{\dot{B}_{p,1}^{\frac{d}{p}}}dt\\
& +C\int_{0}^{T}\|a_{N}\|_{
\dot{B}_{p,1}^{1+\frac{d}{p}}}dt +
C\|a_{N}\|_{L_{T}^{\infty}\dot{B}_{p,1}^{\frac{d}{p}}}
\int_{0}^{T}\|u_{N}\|_{\dot{B}_{p,1}^{2+\frac{d}{p}}}dt.\\
\end{aligned}
\end{equation}
We  multiply \eqref{AA2.6} by $\eta_{0}$ and add it to \eqref{A2.11},
 \begin{equation*}
\begin{aligned}
&\eta_{0}\|a_{N}\|_{L_{T}^{\infty}(\dot{B}_{p,1}^{1+\frac{d}{p}})}+
\|u_{N}\|_{L_{T}^{\infty}(\dot{B}_{p,1}^{\frac{d}{p}})}+\int_{0}^{T}
\|u_{N}\|_{\dot{B}_{p,1}^{2+\frac{d}{p}}}dt
\\
&
\leq C\|P_{\leq N}u_0\|_{\dot{B}_{p,1}^{\frac{d}{p}}}+\int_{0}^{T}(( C\eta_{0}+1)\|\nabla u_{N}\|_{ \dot{B}_{p,1}
^{\frac{d}{p}}}+1)(\|a_{N}\|_{
\dot{B}_{p,1}^{1+\frac{d}{p}}}+\|u_{N}\|_{\dot{B}_{p,1}^{\frac{d}{p}}})dt\\
&\   \   \   \  +\eta_{0}\|P_{\leq N}a_0\|_{\dot{B}_{p,1}^{1+\frac{d}{p}}}+C(\eta_{0}+(\eta_{0}+1)\|a_{N}\|_{L_{T}^{\infty}\dot{B}_{p,1}
^{\frac{d}{p}}})
\int_{0}^{T}\|u_{N}\|_{\dot{B}_{p,1}^{2+\frac{d}{p}}}dt,\\
\end{aligned}
\end{equation*}
here we can choose $\eta_{0}$ small enough. We thus deduce from $(a_{N}, u_{N})\in\mathbb{E}_{uni}$ that $C\eta_{0}+C(\eta_{0}+1)\|a_{N}\|_{L_{T}^{\infty}\dot{B}_{p,1}^{\frac{d}{p}}}\leq\frac{1}{2}$. Then we have
 \begin{equation*}
\begin{aligned}
&\eta_{0}\|a_{N}\|_{L_{T}^{\infty}(\dot{B}_{p,1}^{1+\frac{d}{p}})}
+\|u_{N}\|_{L_{T}^{\infty}(\dot{B}_{p,1}^{\frac{d}{p}})}+\frac{1}{2}\int_{0}^{T}
\|u_{N}\|_{\dot{B}_{p,1}^{2+\frac{d}{p}}}dt
\\
&
\leq C\|P_{\leq N}u_{0}\|_{\dot{B}_{p,1}^{\frac{d}{p}}}+\eta_{0}\|\|P_{\leq N}a_{0}\|_{\dot{B}_{p,1}^{1+\frac{d}{p}}}
\\
&
 \qquad +\int_{0}^{T}((C\eta_{0} +1)\|\nabla u_{N}\|_{ \dot{B}_{p,1}
^{\frac{d}{p}}}+1)(\|a_{N}\|_{
\dot{B}_{p,1}^{1+\frac{d}{p}}}+\|u_{N}\|_{\dot{B}_{p,1}^{\frac{d}{p}}})dt.\\
\end{aligned}
\end{equation*}
Applying the Gronwall inequality and the uniform boundedness of $(a_{N}, u_{N})$ implies that
 \begin{equation}\label{unau}
\begin{aligned}
&\|a_{N}\|_{L_{T}^{\infty}(\dot{B}_{p,1}^{1+\frac{d}{p}})}+\|u_{N}\|_{L_{T}^{\infty}(\dot{B}_{p,1}^{\frac{d}{p}})}+\int_{0}^{T}
\|u_{N}\|_{\dot{B}_{p,1}^{2+\frac{d}{p}}}dt
\\
& \ \ \
\lesssim(\|P_{\leq N}u_{0}\|_{\dot{B}_{p,1}^{\frac{d}{p}}}+\|P_{\leq N}a_{0}\|_{\dot{B}_{p,1}^{1+\frac{d}{p}}})\exp\{\int_{0}^{T}(\|\nabla u_{N}\|_{ \dot{B}_{p,1}
^{\frac{d}{p}}}+1)dt\}
\\
& \ \ \
\leq C(\|P_{\leq N}u_{0}\|_{\dot{B}_{p,1}^{\frac{d}{p}}}+\|P_{\leq N}a_{0}\|_{\dot{B}_{p,1}^{1+\frac{d}{p}}})\leq C2^{N}(\|u_{0}\|_{\dot{B}_{p,1}^{-1+\frac{d}{p}}}+\|a_{0}\|_{\dot{B}_{p,1}^{\frac{d}{p}}}).
\end{aligned}
\end{equation}

 {\bf  Step 2. }
 Using  uniform estimate $(a_{N}, u_{N})\in\mathbb{E}_{uni}$, we infer
 \begin{equation}\label{anu2}
\begin{aligned}
\int_{0}^{T}\|a_{N}\|_{\dot{B}_{p,1}^
   {1+\frac{d}{p}}}^2\mathrm{d}t\lesssim 2^{2N},
  \end{aligned}
 \end{equation}
 from which and  Proposition~\ref{diff-1-local} (case 1) and \eqref{anu2}, we have the following uniform difference estimate
 \begin{equation*}
\begin{aligned}
&\|S_T(P_{\leq N}a_{0},\,P_{\leq N}u_{0})-S_T(P_{\leq N}\tilde{a}_{0},P_{\leq N}\tilde{u}_{0})\|_{Z_p(T)}\\
& \ \ \ \lesssim\exp\big\{\int_0^T\|a_N(\tau)\|_{\dot{B}_{p,1}^
   {1+\frac{d}{p}}}^2\mathrm{d}t\big\}(\|P_{\leq N}\tilde{a}_{0}-P_{\leq N}a_{0}\|_{\dot{B}_{p,1}^{\frac{d}{p}}}
+\|P_{\leq N}\tilde{u}_{0}-P_{\leq N}u_{0}\|_{\dot{B}_{p,1}^{-1+\frac{d}{p}}})
   \\
   & \ \ \ \leq C\exp\big\{C2^{2N}\big\}(\|\tilde{a}_{0}-a_{0}\|_{\dot{B}_{p,1}^{\frac{d}{p}}}
   +\|\tilde{u}_{0}-u_{0}\|_{\dot{B}_{p,1}^{-1+\frac{d}{p}}}).
   \end{aligned}
 \end{equation*}
 Meanwhile, according to Proposition ~\ref{diff-1-local} (case 2) and the uniform estimate \eqref{unau}, we also can get
\begin{align*}
    &\|S_T(a_0, u_{0})-S_T(P_{\leq N}a_0, P_{\leq N}u_0)\|_{Z_p(T)}\\
   & \  \ \ \lesssim \|a_0-P_{\leq N}a_0\|_{\dot{B}_{p,1}^{\frac{d}{p}}}
   +\|u_0-P_{\leq N}a_0\|_{\dot{B}_{p,1}^{-1+\frac{d}{p}}}\\& \qquad +\|a_{N}\|_{L^\infty_T(\dot{B}_{p,1}^{1+\frac{d}{p}})}(\|a_0-P_{\leq N}a_0\|_{\dot{B}_{p,1}^{-1+\frac{d}{p}}}
   +\|u_0-P_{\leq N}u_0\|_{\dot{B}_{p,1}^{-2+\frac{d}{p}}})
   \\
   & \  \ \ \lesssim \|a_0-P_{\leq N}a_0\|_{\dot{B}_{p,1}^{\frac{d}{p}}}
   +\|u_0-P_{\leq N}u_0\|_{\dot{B}_{p,1}^{-1+\frac{d}{p}}}\\& \qquad +2^{N}\|a_0-P_{\leq N}a_0\|_{\dot{B}_{p,1}^{-1+\frac{d}{p}}}
   +2^{N}\|u_0-P_{\leq N}u_0\|_{\dot{B}_{p,1}^{-2+\frac{d}{p}}}\\
   & \  \ \ \leq C\big(\|a_0-P_{\leq N}a_0\|_{\dot{B}_{p,1}^{\frac{d}{p}}}
   +\|u_0-P_{\leq N}u_0\|_{\dot{B}_{p,1}^{-1+\frac{d}{p}}}\big),
  \end{align*}
and similarly, we have
   \begin{align*}
\|S_T(\tilde{a}_0, \tilde{u}_{0})-S_T(P_{\leq N}\tilde{a}_0, P_{\leq N}\tilde{u}_0)\|_{Z_p(T)}\leq C\big(\|\tilde{a}_0-P_{\leq N}\tilde{a}_0\|_{\dot{B}_{p,1}^{\frac{d}{p}}}
   +\|\tilde{u}_0-P_{\leq N}\tilde{u}_0\|_{\dot{B}_{p,1}^{-1+\frac{d}{p}}}\big).
  \end{align*}

 {\bf  Step 3.} Based on the above estimates, we shall demonstrate the continuity of the solution map. We find
\begin{equation}\label{au00}
\begin{split}
& \|S_{T}(a_0,u_0) -S_{T}(\tilde{a}_0, \tilde{u}_0 )\|_{Z_p(T)}\\
& \ \ \  \leq \|S_{T}(a_0,u_0) -S_{T}(P_{\leq N}a_0,P_{\leq N}u_0)\|_{Z_p(T)}
 \\
& \ \ \   \ \ \ +
\|S_{T}(\tilde{a}_0, \tilde{u}_{0}) -S_{T}(P_{\leq N}\tilde{a}_0, P_{\leq N}\tilde{u}_0)\|_{Z_p(T)}\\
& \qquad+ \|S_{T}(P_{\leq N}a_0, P_{\leq N}u_0) -S_{T}(P_{\leq N}\tilde{a}_0, P_{\leq N}\tilde{u}_0)\|_{Z_p(T)} \\
 & \ \ \  \leq C\big( \|(u_0 -P_{\leq N}u_0,\tilde{u}_0 -P_{\leq N}\tilde{u}_0)\|_{{\dot{B}}_{p,1}^{-1+\frac{d}{p}}} + \|(a_0-P_{\leq N}a_0,\tilde{a}_0-P_{\leq N}\tilde{a}_0)\|_{{\dot{B}}_{p,1}^{\frac{d}{p}}} \big)  \\
& \ \ \   \ \ \ +C\exp\big\{C2^{2N}\big\}(\|a_{0}-\tilde{a}_{0}\|_{
\dot{B}_{p,1}^{\frac{d}{p}}}
   +\|u_{0}-\tilde{u}_{0}\|_{\dot{B}_{p,1}^{-1+\frac{d}{p}}})\\
   & \ \ \  := C\|(a_0 -P_{\leq N}a_0, u_0 -P_{\leq N}u_0)\|_{\mathbb{X}_p} 
   +C\|(\tilde{a}_0 -P_{\leq N}\tilde{a}_0, \tilde{u}_0 -P_{\leq N}\tilde{u}_0)\|_{\mathbb{X}_p}\\
& \ \ \   \ \ \ +C\exp\big\{C2^{2N}\big\}(\|(a_{0}-\tilde{a}_{0}, u_{0}-\tilde{u}_{0} )\|_{\mathbb{X}_{p}},
\end{split}
\end{equation}
in the above inequality, we used the results of Step 2.  Given that
$1 \leq p <d$, for an arbitrary
$\varepsilon>0$, we can select
$N$ to be sufficiently large such that
\begin{equation}\label{au0}
\begin{split}
 C \|(a_0 -P_{\leq N}a_0, u_0 -P_{\leq N}u_0)\|_{\mathbb{X}_p} \leq \frac{\varepsilon}{8}.  \\
\end{split}
\end{equation}
Fixing $N$, we can choose  $\tilde{\delta}_{0}\leq\min\{\eta,\frac{\varepsilon}{4}\}$ (here $\eta$ comes from \eqref{Diffau}) to be sufficiently small such that
$ C\|(a_0 -\tilde{a}_{0}, u_{0}-\tilde{u}_{0})\|_{\mathbb{X}_{p}}\leq\tilde{\delta}_{0}$ and 
\begin{equation}\label{au1}
\begin{aligned}
C\|(a_0 -\tilde{a}_{0}, u_{0}-\tilde{u}_{0})\|_{\mathbb{X}_{p}}\exp\big\{C2^{2N}\big\}\leq\frac{\varepsilon}{2}.
\end{aligned}
\end{equation}
Next, fixing $N$ and $\tilde{\delta}_{0}$, we obtain 
\begin{equation}\label{au2}
\begin{aligned}
&C \|(\tilde{a_0 }-P_{\leq N}\tilde{a}_0,\tilde{u}_0 -P_{\leq N}\tilde{u}_0)\|_{\mathbb{X}_p}\\
& \  \  \leq C \|(\tilde{a}_0 -a_0,\tilde{u}_0 -u_0)\|_{\mathbb{X}_p}+C \|(a_0 -P_{\leq N}a_0,  u_0 -P_{\leq N}u_0)\|_{\mathbb{X}_p}\\
& \  \  \  \ +C\|(P_{\leq N}a_0 -P_{\leq N}\tilde{a}_0, P_{\leq N}u_0 -P_{\leq N}\tilde{u}_0)\|_{\mathbb{X}_p}\\
& \  \  \leq C \|(\tilde{a}_0 -a_0,\tilde{u}_0 -u_0)\|_{\mathbb{X}_p}+C \|(a_0 -P_{\leq N}a_0,  u_0 -P_{\leq N}u_0)\|_{\mathbb{X}_p} \leq \frac{3}{8}\varepsilon.\\
\end{aligned}
\end{equation}
Consequently, plugging estimates \eqref{au0}, \eqref{au1} and \eqref{au2} into \eqref{au00}, we have
\[
 \|S_{T}(a_0,u_{0}) -S_{T}(\tilde{a}_0, \tilde{u}_0 )\|_{Z_p(T)}\leq\varepsilon,
\]
which concludes the desired results as stated in Theorem~\ref{apelwp}.


\subsection*{Acknowledgements}
Z. Guo is the recipient of an Australian Research Council Future Fellowship (project number FT230100588) funded by the Australian Government. Yang is supported by the National Natural Science Foundation of China, 12161041. Zhang is supported by the National Natural Science Foundation of China, 11801425.


\footnotesize
\begin{thebibliography}{100}

\bibitem{BaChDa11} Bahouri, H., Chemin, J.Y., Danchin, R.: {F}ourier Analysis and Nonlinear Partial Differential Equations. Springer, Heidelberg (2011).
\bibitem{Bona-Smith} Bona, J.L., Smith, R.: The initial-value problem for the Korteweg-de Vries equation. Philos. Trans.R. Soc. Lond. Ser. A. 278. 555-601 (1975).
\bibitem{Bo81} Bony, J.M.: Calcul symbolique et propagation des singularit\'{e}s pour les \'{e}quations aux d\'{e}riv\'{e}es partielles non lin\'{e}aires. Ann. Sci. \'{E}c. Norm. Sup{\'e}r. (4) {\bf 14}, 209--246 (1981).
\bibitem {FD2010} Charve, F., Danchin, R.: A global existence result for the compressible Navier-Stokes equations in the critical $L^{p}$ framework. Arch. Rational Mech. Anal. 
{\bf 198}, 233--271 (2010).
\bibitem {CMZ2010R} Chen, Q., Miao, C., Zhang, Z.: Well-posedness in critical spaces for the compressible Navier-Stokes equations with density-dependent viscosities. Rev. Mat. Iberoam {\bf 26},  1173--1224 (2010).
\bibitem {CMZ2010} Chen, Q., Miao, C., Zhang, Z.: Global well-posedness for compressible Navier-Stokes equations with highly oscillating initial velocity. Comm. Pure Appl. Math. {\bf 63}, 1173--1224 (2010). 
\bibitem {CMZ2015}Chen, Q., Miao, C., Zhang, Z.: On the ill-posedness of the compressible Navier-Stokes equations. Rev. Mat. Iberoam {\bf 31}, 1375--1402 (2015).
\bibitem {D2001}  Danchin, R.: Local theory in critical spaces for compressible viscous and heat-conductive gases.  Commun. Partial Differ. Equ. {\bf 26}, 1183--1233 (2001).
\bibitem {D2005}Danchin, R.: On the uniqueness in critical spaces for 
compressible Navier–Stokes equations. NoDEA Nonlinear Differ. 
Equ. Appl. {\bf 12(1)}, 111--128 (2005).
\bibitem {D2007}  Danchin, R.: 
Well-posedness in critical spaces for barotropic viscous fluids with truly not constant density. Commun. Partial Differ. Equ. {\bf 32}, 1373--1397 (2007).
\bibitem {D2000} Danchin, R.:  Global existence in critical spaces for compressible Navier-Stokes equations. Invent Math. {\bf 141}, 579--614 (2000).
\bibitem {D2014} Danchin, R.: A Lagrangian approach for the compressible Navier-Stokes equations. Ann. Inst. Fourier {\bf 64(2)}, 753--791 (2014).
\bibitem {RH2015} Danchin, R.: {F}ourier analysis methods for the compressible Navier-Stokes equations. arXiv: 1507. 02637 (2015).
\bibitem {FK1964} Fujita, H., Kato, T.: On the Navier-Stokes initial value problem I. 
Arch. Rational Mech. Anal.  {\bf 16}, 269--315 (1964).
\bibitem{GuLi21}
Guo, Z., Li, K.: Remarks on the well-Posedness of the Euler Equations in the Triebel-Lizorkin Spaces.  J. Fourier Anal. Appl. {\bf 27}, 1-24 (2021).
\bibitem{GuLiYi18} Guo, Z., Li, J., Yin, Z.: Local well-posedness of the incompressible {E}uler equations in {$B^{1}_{\infty,1}$} and the inviscid limit of the {N}avier–{S}tokes equations. J. Funct. Anal. {\bf 276}, 2821--2830 (2018).
\bibitem{GSY} Guo, Z., Song, Z., Yang, M.: Global well-posedness for the 3D compressible Navier-Stokes equations in optimal Besov space, arXiv:2509.17005.
\bibitem {H2011} Haspot, B.:  Existence of global strong solutions in critical spaces for barotropic viscous fluids. Arch. Ration Mech. Anal. {\bf 202}, 427--460 (2011).
\bibitem {IO2022}  Iwabuchi, T., Ogawa, T.: Ill-posedness for the compressible Navier-Stokes equations under barotropic condition in limiting Besov spaces. Journal of the Mathematical Society of Japan {\bf 74(2)}, 353--394 (2022).
\bibitem {KoTz} Koch, H., Tzvetkov, N.: Local well-posedness of the Benjamin-Ono equation in $H^{s}(\mathbb{R})$. Int. Math. Res. Not. {\bf14}, 1449--1464(2003).
\bibitem {PL1998}  Lions, P. L.: Mathematical Topics in Fluid Mechanics: Vol 2: Compressible Models. Oxford: Oxford University Press, 1998.
\bibitem {Tao04} Tao T.: Global well-posedness of the Benjamin-Ono equation in $H^{1}$. Journal of Hyperbolic Differential Equations  {\bf (01)1}, 27--49 (2004).
\bibitem{Tao2}
Tao T.: Global existence and uniqueness results for weak solutions of the focusing mass-critical nonlinear Schr\"odinger equation. Anal. PDE {\bf 2(1)}, 61-81 (2009). DOI: 10.2140/apde.2009.2.61

\end{thebibliography}
\end{document}